\documentclass{jstpip}
\usepackage{amsmath,amsthm,amssymb}

\newcommand{\nexteqv}{\displaybreak[0]\\ &\iff}
\newcommand{\nexteq}{\displaybreak[0]\\ &=}
\newcommand{\nnexteq}{\nonumber\displaybreak[0]\\ &=}
\newtheorem{lem}{Lemma}[section]
\newtheorem{thm}[lem]{Theorem}

\theoremstyle{definition}
\newtheorem{rem}[lem]{Remark}
\newtheorem{dfn}[lem]{Definition}

\newtheorem{exam}[lem]{Example}

\newcommand{\F}{\mathbb{F}}
\newcommand{\Z}{\mathbb{Z}}
\newcommand{\hA}{\hat{A}}

\newcommand{\cQ}{\mathcal{Q}}

\newcommand{\cB}{\mathcal{B}}
\newcommand{\cC}{\mathcal{C}}
\newcommand{\cD}{\mathcal{D}}
\newcommand{\cE}{\mathcal{E}}
\newcommand{\cF}{\mathcal{F}}
\newcommand{\cG}{\mathcal{G}}

\newcommand{\cT}{\mathcal{T}}

\DeclareMathOperator{\Orb}{Orb}

\DeclareMathOperator{\SQS}{SQS}
\newcommand{\orb}[2]{\Orb_{#1}({#2})}
\newcommand{\orbt}[1]{[#1]}
\newcommand{\qbinom}[2]{\genfrac{[}{]}{0pt}{}{#1}{#2}}

\begin{document}

\title{Steiner Quadruple Systems with
Point-Regular Abelian Automorphism Groups}
\begin{authors}
  \author{Akihiro Munemasa}{Tohoku University}{Sendai, 980-8579 Japan}{munemasa@math.is.tohoku.ac.jp}
  \author{Masanori Sawa}{Nagoya University}{Nagoya, 464-8601 Japan}{sawa@math.cm.is.nagoya-u.ac.jp}
\end{authors}
\date{April 14th, 2010}{December 20th, 2010}

\begin{abstract}
In this paper
we present a graph theoretic construction of Steiner quadruple systems (SQS)
admitting abelian groups as point-regular automorphism groups.
The resulting SQS has an extra property which we call
$A$-reversibility, where $A$ is the underlying abelian group.
In particular, when $A$ is a $2$-group of exponent at most $4$,
it is shown that an $A$-reversible SQS always exists.
When the Sylow $2$-subgroup of $A$ is cyclic, we give
a necessary and sufficient condition for the existence of
an $A$-reversible SQS, which is a generalization of
a necessary and sufficient condition for the existence of
a dihedral SQS by Piotrowski (1985).
This enables one to construct $A$-reversible SQS for any
abelian group $A$ of order $v$ such that
for every prime divisor $p$ of $v$
there exists a dihedral $\SQS(2p)$.

\amsc{05E20, 05B05}

\keywords{Steiner system; combinatorial design; graph; finite group.}
\end{abstract}

\section{Introduction}
Let $t,k,\lambda$ and $v$ be positive integers such that $t < k < v$.
A $t$-$(v,k,\lambda)$ {\it design} is an ordered pair $\cD = (V,\cB)$
consisting of a set $V$ of 
$v$ points, and a collection $\cB$ of $k$-subsets (called blocks) of $V$
such that every 
$t$-subset of $V$ occurs in exactly $\lambda$ blocks.
In particular, $\cD$ is a {\it Steiner system} if $\lambda=1$.
A {\it Steiner quadruple system} (SQS) of order $v$,
denoted by $\SQS(v)$, is a $3$-$(v,4,1)$ design. 
A well known result by
\cite{HH}
states that
an $\SQS(v)$ exists if and only if $v\equiv 2$ or $4\pmod{6}$.

An {\it automorphism} of $\cD$ is a permutation $\xi$ on $V$ such that
$B^\xi \in \cB$ for each 
$B \in \cB$. The collection of all automorphisms of $\cD$ forms a group,
called the {\it full automorphism group}, and a subgroup of the full automorphism
group is an {\it automorphism group} 
of $\cD$. A finite group $A$ acting on $V$ induces a natural action on the
set $\binom{V}{k}$
of all $k$-subsets of $V$.  
In a classical method of constructing $t$-designs one chooses suitable 
$A$-orbits of $k$-subsets and 
obviously the resulting design admits $A$ as an automorphism group. 
It is well known 
that 
if $A$ is $t$-transitive on $V$, then there exists a $t$-design with $A$
as an automorphism group. Since
there are few $t$-transitive groups with $t\ge 3$, 
we wish to develop a method of constructing 
$t$-designs, which works for permutation groups
with low transitivity. 

In this paper, we take an abelian group $A$ of order $v$ as
the set of points, and construct $\SQS(v)$ whose set of blocks
has some extra property.
If $A$ is an abelian group, then $A$ acts on itself by translation. 
Let $\sigma$ be the involutory automorphism of $A$ defined by $a^\sigma = -a$. 
We regard $A$ as a permutation group acting on $A$ regularly,
and form the semidirect product $\hA=A\rtimes\langle
\sigma\rangle$. The group $\hA$ is a 
permutation group on $A$.
A subset $B$ of $A$ is said to be {\it symmetric} if 
$B = -B+x$ for some $x\in A$, or equivalently, the orbit of $B$ 
under $\hA$ coincides with
that of $B$ under $A$.
An $\SQS(v)$ on points $A$ is said to be {\it $A$-reversible} if 
every block is symmetric and the set of all blocks
is invariant under the action of $\hA$.
When $A$ is the cyclic group $\Z_v$, then an
$A$-reversible $\SQS(v)$ has been known as an
$S$-cyclic $\SQS(v)$.
It was
\cite{FF}
who
first proposed a graph theoretic construction of $S$-cyclic $\SQS(v)$.
Later, his construction was taken up again and refined by
\cite{EK4},
in which the notion of {\it the first K\"ohler graph of order $v$} was introduced. 
In the same paper, among many other things, K\"ohler proved
a fundamental theorem for $S$-cyclic $\SQS(v)$ such 
that the stabilizer of any quadruple under the action of $\Z_v$ equals the identity,
which states that for $v\equiv 2$ or $10\pmod{24}$, 
if the first K\"ohler graph of order $v$ has a $1$-factor, 
then there exists a $S$-cyclic $\SQS(v)$.
After K\"ohler's work,
some researchers have tried to construct $S$-cyclic $\SQS(v)$, 
on which some progress has been made, but far from settled in general. 
Without restriction of the stabilizers of quadruples, 
\cite{WP}[Satz 14.1]
proved a theorem stating that
there exists a $S$-cyclic $\SQS(v)$ if and only if 
$v\equiv 0\pmod{2}$, $v\not\equiv0\pmod{3}$, $v\not\equiv0\pmod{8}$,
$v\ge4$ and for any prime divisor $p$ of $v$ 
there exists a $S$-cyclic $\SQS(2p)$.
See~\cite{BE1,Feng,HS1,HS2,HS3} for more information on $S$-cyclic $\SQS(v)$. 

The main purpose of this paper is to generalize 
Piotrowski's theorem on $S$-cyclic $\SQS(v)$ to
$A$-reversible $\SQS(v)$.
In Section 2,
the concept of the {\it K\"ohler graph} is introduced
as a generalization of first K\"ohler graphs of cyclic groups.
In Section 3, the structure of K\"ohler graphs is investigated
with special emphasis on degree of vertex and connected component.
In Section 4,
the number of orbits of certain triples and quadruples for an abelian group
are counted.
In Section 5,
we assume that $A$ is an abelian group of order $v\equiv2$ or $4\pmod{6}$.
We show that the special classes of triples and quadruples discussed in Section 4 could be incorporated
into an $A$-reversible $\SQS(v)$, thereby reducing the
existence of an $A$-reversible $\SQS(v)$ to a graph theoretic
problem about the K\"ohler graph of $A$
in Section 6.
As an example,
an infinite family of $A$-reversible $\SQS(2^n)$ is given
for abelian groups $\Z_2^a\times\Z_4^b$.
Finally in Section 7,
we prove that for an abelian group $A$ whose Sylow $2$-subgroup is cyclic, 
the following statements are equivalent:
\begin{enumerate}
\item There exists an $A$-reversible $\SQS(v)$;
\item $v\equiv 0\pmod{2}$, $v\not\equiv0\pmod{3}$,
$v\not\equiv0\pmod{8}$, $v\ge4$ and for any prime
divisor $p$ of $v$
there exists a $S$-cyclic $\SQS(2p)$.
\end{enumerate}
This is a generalization of a
theorem of Piotrowski, reformulated by
\cite{HS1}[p.93].
In~\cite{HS2},
it is shown that a $S$-cyclic $\SQS(2p)$ exists for
any prime number $p\equiv53$ or $77\pmod{120}$ with $p<500000$.
Applying our theorem to this result shows that there exists an
$A$-reversible $\SQS(v)$
for any abelian group $A$ of order $v$ which is twice a
product of prime numbers $p$ with
$p\equiv53$ or $77\pmod{120}$
and $p<500000$. 

\setcounter{section}{1}
\section{The K\"ohler graph of an abelian group}
Throughout this section, we
let $A$ be an abelian group
of order $v$.
We regard $A$ as a permutation group acting on $A$ regularly,
and form the semidirect product $\hA=A\rtimes\langle
\sigma\rangle$, where $\sigma$ is the automorphism of
$A$ defined by $a^\sigma=-a$. The group $\hA$ is a 
permutation group on $A$.
For a subset $X$ of $A$, let $\orb{\hA}{X}$ denote the
$\hA$-orbit of $X$:
\[
\orb{\hA}{X}=\{X+a\mid a\in A\}\cup\{-X+a\mid a\in A\}.
\]
For distinct nonzero elements $a_1,\dots,a_t\in A$, 
we abbreviate
\[
\orb{\hA}{\{0,a_1,\dots,a_t\}}
\]
as $\orbt{a_1,\dots,a_t}$.
If $\{0,a,b\}\in\binom{A}{3}$, then
\begin{equation}\label{eq:T0}
\begin{split}
&\{X\mid \{0\}\cup X\in\orbt{a,b},\;0\notin X\}
\\ &=
\{\{a,b\},\{-a,b-a\},\{-b,a-b\},
\\ &\qquad
\{-a,-b\},\{a,a-b\},\{b,b-a\}\}.
\end{split}
\end{equation}
If $\{0,a,b,a+b\}\in\binom{A}{4}$, then
\begin{equation}\label{eq:B0a}
\begin{split}
&\{Y\mid \{0\}\cup Y\in\orbt{a,b,a+b},\;0\notin Y\}
\\ &=
\{\{a,b,a+b\},\{-a,b,-a+b\},
\\ &\qquad
\{a,-b,a-b\},\{-a,-b,-a-b\}\}.
\end{split}
\end{equation}
Let 
\begin{equation}
\begin{split}
\cT&=
\{ \orbt{a,b} \mid a,b\in A,\;a\neq\pm b,
\\ &\qquad
2a\notin\{0,b,2b\},\;2b\notin\{0,a,2a\}\},
\label{eq:defT}
\end{split}
\end{equation}
\begin{equation}
\begin{split}
\cE &= \{ \orbt{a,b,a+b} \mid a,b\in A,\;
0\notin\{2a,2b\},
\\ &\qquad
\{\pm a,\pm 2a\}\cap\{\pm b,\pm 2b\}=\emptyset\}.
\label{eq:defE}
\end{split}
\end{equation}

\begin{dfn}\label{dfn:K}
The {\em K\"ohler graph} of $A$ is the incidence structure
$\cG=(\cT,\cE)$, where $\cT,\cE$ are defined in 
(\ref{eq:defT}), (\ref{eq:defE}), respectively,
and $\orb{\hA}{T}\in\cT$ is incident with $\orb{\hA}{B}
\in\cE$ if $B\supset T'$ for some $T'\in\orb{\hA}{T}$.
\end{dfn}

\begin{lem}\label{lem:cE}
\begin{enumerate}
\item
If $\{0,a,b\}\in\binom{A}{3}$, then
$\orbt{a,b}\in\cT$ if and only if
\begin{equation}\label{eq:abT}
a\neq\pm b,\quad 
2a\notin\{0,b,2b\}\text{ and }
2b\notin\{0,a,2a\}.
\end{equation}
\item
If $\{0,a,b,a+b\}\in\binom{A}{4}$, then
$\orbt{a,b,a+b}\in\cE$ if and only if
\begin{equation}\label{eq:abE}
0\notin\{2a,2b\}\text{ and }
\{\pm a,\pm2a\}\cap\{\pm b,\pm2b\}=\emptyset.
\end{equation}
\end{enumerate}
\end{lem}
\begin{proof}
(i)
The ``if'' part is trivial. To prove the ``only if'' part,
suppose that there exist
$c,d\in A$ such that $\orbt{a,b}=\orbt{c,d}$ and
\begin{equation}\label{eq:cdE1}
c\neq\pm d,\;
2c\notin\{0,d,2d\},\;2d\notin\{0,c,2c\}.
\end{equation}
By (\ref{eq:T0}), we have
$
\{a,b\}\in\{\pm\{c,d\},\pm\{c,c-d\},\pm\{d,d-c\}\}
$.
Thus, either (\ref{eq:abT}) holds, or
$\{a,b\}\in\{\pm\{c,c-d\},\pm\{d,d-c\}\}$ holds.
In the latter case, replacing $(c,d)$ by $(-c,-d)$,
$(d,c)$ or $(-d,-c)$ if necessary, we may assume
$(a,b)=(c,c-d)$. Then by (\ref{eq:cdE1}), we have
$
a\neq\pm(a-b),\;2a\notin\{0,a-b,2(a-b)\},\;
2(a-b)\notin\{0,a,2a\}
$, 
and (\ref{eq:abT}) holds also in this case.

(ii)
The ``if'' part is trivial. To prove the ``only if'' part,
suppose that there exist
$
c,d\in A
$
such that $\orbt{a,b,a+b}=\orbt{c,d,c+d}$ and
\begin{equation}\label{eq:cdE2}
0\notin\{2c,2d\},\;
\{\pm c,\pm2c\}\cap\{\pm d,\pm2d\}=\emptyset.
\end{equation}
By (\ref{eq:B0a}), we have
$
\{a,b,a+b\}\in\{ \pm\{c,d,c+d\}, \pm\{-c,d,-c+d\}\}
$.
Replacing $(c,d)$ by $(-c,-d)$, $(-c,d)$ or $(c,-d)$
if necessary, we may assume
$\{a,b,a+b\}=\{c,d,c+d\}$.
This implies
$\{a,b\}=\{c,d\}$,
and (\ref{eq:abE}) holds by (\ref{eq:cdE2}). 
\end{proof}

\begin{lem}\label{lem:onlyB}
Suppose $T\in\binom{B}{3}$ and $\orb{\hA} {B}\in\cE$. Then $B$ is
the only member of $\orb{\hA}{B}$ containing $T$.
\end{lem}
\begin{proof}
Without loss of generality we may assume $B=\{0,a,b,a+b\}$,
with $a,b\in A$ satisfying (\ref{eq:abE}).
We first suppose $T=\{0,a,b\}$. 
Then
\begin{align*}
&|\{B'\in\orb{\hA}{B}\mid T\subset B'\}|
\nexteq
|\{B'\in\orbt{a,b,a+b}\mid T\subset B'\}|
\nexteq
|\{C\in\{\{a,b,a+b\},\{-a,b-a,b\},\{-b,a-b,a\},
\\ &\qquad\qquad
\{-b,-a,-a-b\}\}
\mid \{a,b\}\subset C\}|
\nexteq
|\{\{a,b,a+b\}\}|
\nexteq
1.
\end{align*}

Next suppose $T=\{a,b,a+b\}$. Then $T-(a+b)=
\{0,-a,-b\}$ and $B-(a+b)=\{0,-a,-b,-a-b\}$. By
the first case, we see that $B-(a+b)$ is 
the only member of $\orb{\hA}{B}$ containing $T-(a+b)$.
This gives the desired result.

Next suppose $T=\{0,a,a+b\}$ or $\{0,b,a+b\}$.
Switching $a$ and $b$ if necessary, we may assume
$T=\{0,a,a+b\}$. Then
$T-a=\{0,-a,b\}$ and $B-a=\{0,-a,b,b-a\}$. 
By the first case, we see that $B-a$ is 
the only member of $\orb{\hA}{B}$ containing $T-a$.
This gives the desired result.
\end{proof}

\begin{lem}\label{lem:orbT}
Let $\orbt{a,b}\in\cT$. Then
\begin{enumerate}
\item
$\orbt{a,b}\notin\{\orbt{a,a+b},\orbt{a,b-a},\orbt{b,a-b}\}$,
\item
$\orbt{a,b-a}\neq\orbt{b,a-b}$,
\item $\orbt{a,a+b}=\orbt{a,b-a}$ if and only if 
$(2a+b,5a)=(0,0)$.
\item $\orbt{a,a+b}=\orbt{b,a-b}$ if and only if 
$(a+2b,5b)=(0,0)$.
\end{enumerate}
\end{lem}
\begin{proof}
Since $\orbt{a,b}\in\cT$, Lemma~\ref{lem:cE}(i) implies
that (\ref{eq:abT}) holds. 
If $\orbt{a,b}\in\{\orbt{a,a+b},\orbt{a,b-a}\}$, then
by (\ref{eq:T0}), we have
\begin{align*}
&\{\{a,b\},\{-a,b-a\},\{-b,a-b\},
\{-a,-b\},\{a,a-b\},\{b,b-a\}\}
\\ &\cap
\{\{a,a+b\},\{a,b-a\}\}
\neq\emptyset.
\end{align*}
By (\ref{eq:abT}), we have
$a\notin\{b,-a,b-a,-b,a-b\}$,
and so
$\{a+b,b-a\}\cap\{b,a-b\}\neq\emptyset$,
which is impossible by (\ref{eq:abT}).
Switching the role of
$a$ and $b$, we obtain $\orbt{a,b}\neq\orbt{b,a-b}$.
This establishes (i).

Since
$b\notin\{a,b-a,-a,b-2a,a-b,2a-b\}$ by (\ref{eq:abT}),
we have
\begin{align*}
\{b,a-b\}&\notin
\{\{a,b-a\},\{-a,b-2a\},\{a-b,2a-b\}
\\ &\qquad
\{-a,a-b\},\{a,2a-b\},\{b-a,b-2a\}\}.
\end{align*}
Thus
$\orbt{a,b-a}\neq\orbt{b,a-b}$ by (\ref{eq:T0}).
This proves (ii).

Also, we have
\begin{align*}
&\orbt{a,a+b}=\orbt{a,b-a}
\\ &\iff
\{a,a+b\}\in
\{\{a,b-a\},\{-a,b-2a\},
\\ &\qquad\qquad\qquad\qquad\quad
\{a-b,2a-b\},\{-a,a-b\},
\\ &\qquad\qquad\qquad\qquad\quad
\{a,2a-b\},\{b-a,b-2a\}\}
&&\text{(by (\ref{eq:T0}))}
\displaybreak[0]\\ &\iff
(a,a+b)=(b-2a,-a)
&&\text{(by (\ref{eq:abT}))}
\displaybreak[0]\\ &\iff
b=3a=-2a,
\end{align*}
establishing (iii).

Finally, since
$\orbt{a,a+b}=\orb{\hA}{-\{0,a,a+b\}+(a+b)}=\orbt{b,a+b}$,
(iv) follows from (iii).
\end{proof}

The next lemma shows that the K\"ohler
graph is indeed a graph, in the sense that every member
of $\cE$ is incident with exactly two members of $\cT$.
In~\cite{MS} the concept of K\"ohler graph is already introduced for
an abelian group with cyclic Sylow $2$-subgroup.

\begin{lem}\label{lem:blks}
If $\orbt{a,b,a+b}\in\cE$, then
\begin{equation}
\{\orb{\hA}{T}\mid T\subset\{0,a,b,a+b\},\;|T|=3\}
=
\{ \orbt{a,b},\orbt{a,a+b} \}.
\end{equation}
In particular,
$\orbt{a,b,a+b}$ 
is incident with exactly
two distinct members $\orbt{a,b}$, $\orbt{a,a+b}$ of $\cT$.
\end{lem}
\begin{proof}
Since $\{0,a,a+b\}=-\{0,b,a+b\}+(a+b)$,
we have
\begin{equation}\label{eq:blks3}
\orbt{a,a+b}=\orbt{b,a+b}.
\end{equation}
Since $\{0,a,b\}=-\{a,b,a+b\}+(a+b)$,
we have
\begin{equation}\label{eq:blks4}
\orbt{a,b}=\orb{\hA}{\{a,b,a+b\}}.
\end{equation}
Thus
\begin{align*}
&\{\orb{\hA}{T}\mid T\subset\{0,a,b,a+b\},\;|T|=3\}
\\ &=
\{ \orbt{a,b}, \orbt{a,a+b},
\orbt{b,a+b},\orb{\hA}{\{a,b,a+b\}}\}
\nexteq
\{ \orbt{a,b},\orbt{a,a+b} \}
&&\text{(by (\ref{eq:blks3}), (\ref{eq:blks4})).}
\end{align*}

Since $\orbt{a,b,a+b}\in\cE$,
Lemma~\ref{lem:cE} implies
\begin{align}
&a\neq\pm b,\; 2a\notin\{0,b,2b\},\;
2b\notin\{0,a,2a\}\label{eq:cE1}\\
&a\neq\pm(a+b),\;2a\notin\{0,a+b,2(a+b)\},\;
2(a+b)\notin\{0,a,2a\}.\label{eq:cE2}
\end{align}
By (\ref{eq:cE1}), (\ref{eq:cE2}), we have
$\orbt{a,b}\in\cT$,
$\orbt{a,a+b}\in\cT$,
respectively.
Therefore, the elements of $\cT$ which are
incident with $\orbt{a,b,a+b}$ are
$\orbt{a,b}$ and $\orbt{a,a+b}$.
By Lemma~\ref{lem:orbT}(i), we have
$\orbt{a,b}\neq\orbt{a,a+b}$.
\end{proof}

Finally, we show that the K\"ohler graph has no multiple
edges. To do this, we first determine the edges incident with
a given vertex.

\begin{lem}\label{lem:2.6}
If $\orbt{a,b}\in\cT$, then the edges incident with $\orbt{a,b}$
are
\begin{equation}\label{eq:2.6}
\{\orbt{a,b,a+b},\orbt{a,b,b-a},\orbt{a,b,a-b}\}\cap\cE.
\end{equation}
\end{lem}
\begin{proof}
Observe that $\orbt{a,b}$ is incident with an orbit $\orb{\hA}{B}$ if 
and only if $\{0,a,b\}\subset B'$ for some $B'\in\orb{\hA}{B}$.
In this case, as $0\in B'$, we may assume $B'=\{0,c,d,c+d\}$
for some $c,d\in A$ by (\ref{eq:B0a}). Then
\[
\{a,b\}\in\{\{c,d\},\{c,c+d\},\{d,c+d\}\},
\]
hence
\[
\{c,d,c+d\}\in\{\{a,b,a+b\},\{a,b,b-a\},\{a,b,a-b\}\}.
\]
Thus 
\[
\orb{\hA}{B}=\orb{\hA}{B'}
\in\{\orbt{a,b,a+b},\orbt{a,b,b-a},\orbt{a,b,a-b}\}.
\]

Conversely, every member of (\ref{eq:2.6}) is incident with
$\orbt{a,b}$.
\end{proof}

We conclude this section by a remark.
The K\"ohler graph $\cG=(\cT,\cE)$ is defined as an incidence
structure, so it is nontrivial to prove that $\cG$
has no multiple edges. 
Suppose that a vertex $\orbt{a,b}\in\cT$ is incident with
multiple edges. By Lemma~\ref{lem:2.6}, the possible edges
incident with $\orbt{a,b}$ are
\begin{equation}\label{eq:L2.7a}
\orbt{a,b,a+b},\orbt{a,b,b-a},\orbt{a,b,a-b},
\end{equation}
which are incident with the vertices
\begin{equation}\label{eq:L2.7b}
\orbt{a,a+b},\orbt{a,b-a},\orbt{b,a-b},
\end{equation}
respectively, by Lemma~\ref{lem:blks}. 
By Lemma~\ref{lem:orbT}(ii), we have $\orbt{b,a-b}\neq
\orbt{a,b-a}$, hence the pair $\orbt{a,b,b-a},\orbt{a,b,a-b}$
does not form a pair of multiple edges sharing the common
endpoints. Thus
\begin{equation}\label{eq:L2.7c}
\orbt{a,a+b}=\orbt{a,b-a}\text{ or }\orbt{b,a-b},
\end{equation}
and $\orbt{a,b,a+b}$ is one of the multiple edges.
In particular, $\orbt{a,b,a+b}\in\cE$, and hence
$2a+b\neq0$ and $a+2b\neq0$ by Lemma~\ref{lem:cE}(ii).
Then by Lemma~\ref{lem:orbT}(iii)--(iv), we have
$\orbt{a,a+b}\neq\orbt{a,b-a},\orbt{b,a-b}$, contradicting
(\ref{eq:L2.7c}).
Therefore, the K\"ohler graph $\cG$
has no multiple edges. 

\section{The structure of the K\"ohler graphs}

\begin{lem}\label{lem:BT}
Let $\orbt{a,b}\in\cT$. 
\begin{enumerate}
\item $\orbt{a,a+b}\in\cT$ if and only if
$0\notin\{2a+b,a+2b,2a+2b\}$,
\item $\orbt{a,b-a}\in\cT$ if and only if
$0\notin\{3a-b,3a-2b,4a-2b\}$.
\end{enumerate}
\end{lem}
\begin{proof}
(i)
\begin{align*}
&\orbt{a,a+b}\in\cT
\\ &\iff
a\neq\pm(a+b),\;2a\notin\{0,a+b,2(a+b)\},
\\ &\qquad\quad
2(a+b)\notin\{0,a,2a\}
&&\text{(by Lemma~\ref{lem:cE}(i))}
\\ &\iff
0\notin\{2a+b,a+2b,2a+2b\}
&&\text{(by (\ref{eq:abT})).}
\end{align*}
(ii)
\begin{align*}
&\orbt{a,b-a}\in\cT
\\ &\iff
a\neq\pm(b-a),\;2a\notin\{0,b-a,2(b-a)\},
\\ &\qquad\quad
2(b-a)\notin\{0,a,2a\}
&&\text{(by Lemma~\ref{lem:cE}(i))}
\\ &\iff
0\notin\{3a-b,3a-2b,4a-2b\}
&&\text{(by (\ref{eq:abT})).}
\end{align*}
\end{proof}

For $v \in \cT$ we denote the set of neighbors of $v$ by
$N(v)$.
\begin{lem}\label{lem:nbs}
Let $\orbt{a,b}\in\cT$.
Then
\begin{equation}\label{eq:L3.2a}
N(\orbt{a,b}) = \{\orbt{a,a+b},\orbt{b,a-b},\orbt{a,b-a}\}\cap\cT.
\end{equation}
\end{lem}
\begin{proof}
By Lemmas \ref{lem:blks} and \ref{lem:2.6},
$N(\orbt{a,b})$ is
contained in (\ref{eq:L3.2a}).

Conversely, 
\begin{align*}
&\orbt{a,a+b}\in\cT
\\ &\iff
0\notin\{2a+b,a+2b,2a+2b\}
&&\text{(by Lemma~\ref{lem:BT}(i))}
\\ &\iff
\orbt{a,b,a+b}\in\cE
&&\text{(by Lemma~\ref{lem:cE}(ii))}
\\ &\implies
\orbt{a,a+b}
\in N(\orbt{a,b})
&&\text{(by Lemma \ref{lem:blks}),}
\end{align*}
and
\begin{align*}
&\orbt{a,b-a}\in\cT
\\ &\iff
0\notin\{3a-b,3a-2b,4a-2b\}
&&\text{(by Lemma~\ref{lem:BT}(i))}
\\ &\iff
\orbt{a,b,b-a}\in\cE
&&\text{(by Lemma~\ref{lem:cE}(ii)).}
\\ &\implies
\orbt{a,b-a}
\in N(\orbt{a,b})
&&\text{(by Lemma \ref{lem:blks}).}
\end{align*}
Finally, switching
$a$ and $b$, we see that $\orbt{b,a-b}\in\cT$
implies
$\orbt{b,a-b} \in N(\orbt{a,b})$.
\end{proof}

\begin{lem}\label{lem:orbT3}
Let $\orbt{a,b}\in\cT$. Then the degree of $\orbt{a,b}$
is $3$ if and only if
\[
0\notin\{2a+b,a+2b,2a+2b,3a-b,3a-2b,4a-2b,
3b-a,3b-2a,4b-2a\}.
\]
\end{lem}
\begin{proof}
The degree of $\orbt{a,b}$ is at most $3$ by
Lemma~\ref{lem:nbs}. 
The degree is exactly $3$ if and only if
$\orbt{a,a+b},\orbt{a,b-a}$ and $\orbt{b,a-b}$ are
distinct elements of $\cT$.
By Lemma~\ref{lem:orbT}(ii)--(iv),
$\orbt{a,a+b},\orbt{a,b-a}$ and $\orbt{b,a-b}$ are
distinct if and only if
$(2a+b,5a) \ne (0,0)$ and $(a+2b,5b) \ne (0,0)$.
Therefore,
\begin{align*}
&\text{degree of }\orbt{a,b}\text{ is not }3
\displaybreak[0]\\ &\iff
(2a+b,5a)=(0,0)\text{ or }(a+2b,5b)=(0,0)
\text{ or }\\ &\qquad\quad
\{\orbt{a,a+b},\orbt{a,b-a},\orbt{b,a-b}\}
\not\subset\cT
\displaybreak[0]\\ &\iff
(2a+b,5a)=(0,0)\text{ or }(a+2b,5b)=(0,0)
\text{ or }\\ &\qquad\quad
0\in\{2a+b,a+2b,2a+2b\}
\text{ or }\\ &\qquad\quad
0\in\{3a-b,3a-2b,4a-2b\}
\text{ or }\\ &\qquad\quad
0\in\{3b-a,3b-2a,4b-2a\}
&&\text{(by Lemma~\ref{lem:BT})}
\displaybreak[0]\\ &\iff
0\in\{2a+b,a+2b,2a+2b\}
\\ &\qquad\qquad\quad
\cup\{3a-b,3a-2b,4a-2b\}
\\ &\qquad\qquad\quad
\cup\{3b-a,3b-2a,4b-2a\}.
\end{align*}
\end{proof}

\begin{exam}\label{exam:Z4Z4}
The K\"ohler graph of $\Z_4^2$ is the $3$-cube.
Indeed,
let $A = \langle g_1 \rangle \oplus \langle g_2 \rangle \simeq \Z_4^2$.
The K\"ohler graph of $A$ has
$8$ vertices
\begin{equation}\label{eq:3cube}
v_i =
\left\{\begin{array}{cc}
\orbt{g_1,(i-1)g_1+g_2} & \text{if $i=1,2,3,4$}; \\
\orbt{g_1+2g_2,(i-4)g_1+(2i-7)g_2} & \text{if $i=5,6,7,8$}.
\end{array} \right.
\end{equation}
where
the neighbors $N(v_i)$ of $v_i$ are determined as follows:
\begin{align*}
N(v_1) = \{v_2,v_4,v_5\},\; N(v_2) = \{v_1,v_3,v_8\},\\
N(v_3) = \{v_2,v_4,v_7\},\; N(v_4) = \{v_1,v_3,v_6\},\\
N(v_5) = \{v_1,v_8,v_6\},\; N(v_6) = \{v_4,v_7,v_5\},\\
N(v_7) = \{v_3,v_8,v_6\},\; N(v_8) = \{v_2,v_7,v_5\}.\\
\end{align*}
\end{exam}

\begin{lem}\label{lem:order7or8}
Let $\orbt{a,b}\in\cT$. Then $\orbt{a,b}$
is an isolated vertex of $\cG$
if and only if one of the following conditions holds:
\begin{enumerate}
\item $\orbt{a,b}=\orbt{a,3a}$ with $0\in \{7a,8a\}$, or
$\orbt{a,b}=\orbt{3b,b}$ with $0\in \{7b,8b\}$;
\item $\orbt{a,b}=\orbt{a,-a+h}$ with $6a=0$ and 
$2h=0$.
\end{enumerate}
\end{lem}
\begin{proof}
By Lemma~\ref{lem:nbs}, $\orbt{a,b}$ is an
isolated vertex of $\cG$ if and only if
\[
\orbt{a,a+b}\notin\cT,\;
\orbt{b,a-b}\notin\cT,\;
\orbt{a,b-a}\notin\cT.
\]
By Lemma~\ref{lem:BT}, this occurs precisely when
\begin{align*}
0\in&\{2a+b,a+2b,2a+2b\}\cap
\{3a-b,3a-2b,4a-2b\}
\\ &\cap
\{3b-a,3b-2a,4b-2a\}.
\end{align*}
A tedious computation establishes the desired result.
\end{proof}

\begin{exam}\label{exam:Z7Z8}
Let $A=\langle a\rangle$ be the cyclic group of order
$7$ or $8$. Then by Lemma~\ref{lem:order7or8},
the K\"ohler graphs of $\Z_7$, $\Z_8$ both consist of
a single vertex $\orbt{a,3a}$.
\end{exam}

\begin{lem}\label{lem:2gen}
Suppose that  $\orbt{a,b}\in\cT$ and $\orbt{c,d}\in\cT$
belong to the same connected component of $\cG$.
Then $\langle a,b\rangle=\langle c,d
\rangle$.
\end{lem}
\begin{proof}
It suffices to prove the assertion when $\orbt{c,d}=
\orbt{a,b}$ or $\orbt{c,d}$ is adjacent to $\orbt{a,b}$.
In the former case, the result follows immediately
from (\ref{eq:T0}). In the latter case,
the result follows from Lemma~\ref{lem:nbs}.
\end{proof}

\begin{lem}\label{lem:2gen2}
Let $A'$ be a subgroup of $A$. Then the K\"ohler graph
of $A'$ is isomorphic to a union of connected components
of $\cG$. In particular,
every connected component of the K\"ohler graph of $A$ is 
isomorphic to a connected component of the K\"ohler
graph of a subgroup of $A$ generated by two elements.
\end{lem}
\begin{proof}
For distinct nonzero elements $a_1,\dots,a_t\in A'$, 
we abbreviate
\[
\orb{\hat{A'}}{\{0,a_1,\dots,a_t\}}=\orbt{a_1,\dots,a_t}'.
\]
Let $\cT(A')$, $\cE(A')$ denote the sets
defined by (\ref{eq:defT}), (\ref{eq:defE}), respectively,
for the group $A'$.
By (\ref{eq:T0}) and Lemma~\ref{lem:cE}(i), we have an injective
mapping $\phi:\cT(A')\to\cT$ defined by $\phi(\orbt{a,b}')=
\orbt{a,b}$. 
We claim that $\phi$ is an isomorphism from the K\"ohler graph
of $A'$ onto a union of some connected components of $\cG$. 
Indeed, for $\orbt{a,b}'\in\cT(A')$, Lemma~\ref{lem:nbs} implies
that the neighbors of
$\orbt{a,b}'$ are 
\begin{equation}\label{eq:L3.8a}
\{\orbt{a,a+b}',\orbt{a,b-a}',\orbt{b,a-b}'\}\cap\cT(A').
\end{equation}
The mapping $\phi$ sends (\ref{eq:L3.8a}) to
(\ref{eq:L3.2a}) which is the set of neighbors of $\orbt{a,b}$
in $\cG$ by Lemma~\ref{lem:nbs}. This proves the claim.

Let $\cC\subset\cT$ be a connected component of $\cG$, and let
$\orbt{a,b}\in\cC$. Take $A'$ to be the subgroup of $A$
generated by $a,b$. Then Lemma~\ref{lem:2gen} implies 
$\cC\subset\phi(\cT(A'))$. Since $\phi$ is an isomorphism,
$\cC$ is isomorphic to a connected component of the K\"ohler
graph of $A'$.
\end{proof}

\begin{lem}\label{lem:edgecyc}
Suppose $a,b\in A$ satisfy
\begin{equation}\label{eq:2ab}
2a\notin\langle b\rangle\text{ and }
2b\notin\langle a\rangle.
\end{equation}
Then there exists a cycle containing the edge
$\orbt{a,b,a+b}$ in the K\"ohler graph of $A$.
\end{lem}
\begin{proof}
Let
\[
\cC= \{\orbt{a,ta+b,(t+1)a+b} \mid t \in \Z\}.
\]
We claim that for any integer $t$, 
\begin{equation}\label{eq:orbBt}
\orbt{a,ta+b,(t+1)a+b} \in \cE.
\end{equation}
By (\ref{eq:defE}), this is equivalent to
\[
0\notin \{2a,2(ta+b)\}\text{ and }
\{a,2a\}\cap \{\pm(ta+b),\pm2(ta+b)\} = \emptyset,
\]
whose validity follows from the assumption (\ref{eq:2ab}).
This establishes (\ref{eq:orbBt}). 

Observe that, for $s,t\in\Z$, we have
\begin{align*}
&\orbt{a,ta+b}=\orbt{a,sa+b}
\\ &\iff
\{a,ta+b\}\in\{\{a,sa+b\},\{-a,(s-1)a+b\},
\\ &\qquad\qquad\qquad\qquad
 \{-sa-b,(1-s)a-b\},\{-a,-sa-b\},
\\ &\qquad\qquad\qquad\qquad
\{a,(1-s)a-b\},\{sa+b,(s-1)a+b\}\}
&&\text{(by (\ref{eq:T0}))}
\nexteqv
ta=sa
&&\text{(by (\ref{eq:2ab})).}
\end{align*}
Hence, by Lemma~\ref{lem:blks}, 
\[
\{\orbt{a,ta+b} \mid 0 \le t< |\langle a\rangle|\}
\]
is a set of $|\langle a\rangle|$ distinct vertices of $\cG$ 
which form a cycle of length $|\langle a \rangle|\ge3$. 
\end{proof}

\section{Special orbits of triples and quadruples}
Throughout this section, we
let $A$ be an abelian group
of order $v\equiv2$ or $4\pmod6$.
Let
\begin{align*}
\Omega_1(A)&=\{a\in A\mid 2a=0\},
&\omega_1&=|\Omega_1(A)|,\\
\Omega_2(A)&=\{a\in A\mid 4a=0\},
&\omega_2&=|\Omega_2(A)|.
\end{align*}
Let $\qbinom{\Omega_1(A)}{2}$ denote the set of all
subgroups of order $4$ in $\Omega_1(A)$. Since
$\Omega_1(A)$ is an elementary abelian $2$-group,
$\Omega_1(A)$ can be regarded as a vector space over 
the finite field $\F_2$ of two elements.
Then $\qbinom{\Omega_1(A)}{2}$ is just the set
of $2$-dimensional subspaces of $\Omega_1(A)$, and by
\cite{Dem}[p.28],
\begin{equation}\label{eq:qbin}
\left|\qbinom{\Omega_1(A)}{2}\right|
=\frac{(\omega_1-1)(\omega_1-2)}{6}.
\end{equation}

Let
\begin{align}
\cT_1&=\{\orbt{a,-a}\mid a\in A\setminus\Omega_1(A)\},
\label{eq:T1}\\
\cT_2&=\{\orbt{a,h}\mid 
a\in A\setminus\{0\},\;h\in\Omega_1(A)\setminus
\{0,a\} \}.\label{eq:T2}
\end{align}
By (\ref{eq:T0}), if $b=-a\neq a$, then
\begin{equation}\label{eq:R1}
\{X\mid \{0\}\cup X\in\orbt{a,-a},\;0\notin X\}=
\{\{a,-a\},\{a,2a\},\{-a,-2a\}\},
\end{equation}
and if $a\neq 0$ and $h\in\Omega_1(A)\setminus\{0,a\}$, then
\begin{equation}\label{eq:SS}
\{X\mid \{0\}\cup X\in\orbt{a,h},\;0\notin X\}=
\{\pm\{a,h\},\pm\{a,h+a\},\pm\{h,h+a\}\}.
\end{equation}

Fix an element $h_0$ of order $2$ in $A$. Set
\begin{align}
\cQ_1 &= \{\orbt{a,-a,h_0}\mid a\in A\setminus\Omega_1(A)\},
\label{eq:Q1}\\
\cQ_2 &= \{\orbt{a,h,h+a} \mid a\in A\setminus\Omega_1(A)\;,
h\in\Omega_1(A)\setminus\langle h_0\rangle,\;2a\neq h\},
\label{eq:Q2}\\
\cQ_3 &= \{\orbt{h,h',h+h'}\mid h,h'\in 
\Omega_1(A)\setminus\{0\},\; h\neq h'\}.
\label{eq:Q3}
\end{align}
If $\{0,a,-a,h\}\in\binom{A}{4}$ and $h\in\Omega_1(A)$, then
\begin{equation}\label{eq:B0b}
\begin{split}
&\{Y\mid \{0\}\cup Y\in\orbt{a,-a,h},\;0\notin Y\}
\\ &=
\{\{a,-a,h\},\{h+a,h-a,h\},
\\ &\qquad
\{a,2a,h+a\},\{-a,-2a,h-a\}\}.
\end{split}
\end{equation}

\begin{lem}\label{lem:T}
We have $\binom{A}{3}/\hA=\cT_1\cup\cT_2\cup\cT$,
and $\cT\cap(\cT_1\cup\cT_2)=\emptyset$,
where $\binom{A}{3}/\hA$ is the orbit
decomposition of $\binom{A}{3}$ under $\hA$.
\end{lem}
\begin{proof}
Let $a,b\in A\setminus\{0\}$ be distinct. Then by 
(\ref{eq:R1}), 
\begin{align*}
\orbt{a,b}\in\cT_1
&\iff
\exists c\in A\text{ s.t. }
\{a,b\}\in\{\{c,-c\},\{c,2c\},\{-c,-2c\}\}
\displaybreak[0]\\ &\iff
a=-b,\;a=2b\text{ or }2a=b,
\end{align*}
and by (\ref{eq:SS}),
\begin{align*}
\orbt{a,b}\in\cT_2
&\iff
\exists c\in A,\;h\in\Omega_1(A)\text{ s.t. }
\\ &\qquad\qquad
\{a,b\}\in\{\pm\{c,h\},\pm\{h,h+c\},\pm\{c,h+c\}\}
\displaybreak[0]\\ &\iff
2a=0,\;2b=0\text{ or }2a=2b.
\end{align*}
Combining these we conclude $\orbt{a,b}\notin\cT_1\cup\cT_2$
if and only if $\orbt{a,b}\in\cT$.
\end{proof}

\begin{lem}\label{lem:R1}
Let $a,b\in A\setminus\Omega_1(A)$. Then
\begin{enumerate}
\item
$\orbt{a,-a}=\orbt{b,-b}$
if and only if
$a=\pm b$,
\item
$\orbt{a,-a,h_0}=\orbt{b,-b,h_0}$
if and only if
$b\in\{\pm a,h_0\pm a\}$.
\end{enumerate}
\end{lem}
\begin{proof}
(i) By (\ref{eq:R1}),
\begin{align*}
\orbt{a,-a}=\orbt{b,-b}
&\iff
\{b,-b\}\in\{\{a,-a\},\{a,2a\},\{-a,-2a\}\}
\\ &\iff
a=\pm b,
\end{align*}
since $v\neq0\pmod{3}$.

(ii)
By (\ref{eq:B0b}),
\begin{align*}
&\orbt{a,-a,h_0}=\orbt{b,-b,h_0}
\\ &\iff
\{b,-b,h_0\}\in\{\{a,-a,h_0\},\{h_0+a,h_0-a,h_0\},
\\ &\qquad\qquad\qquad\qquad\quad
\{a,2a,h_0+a\},\{-a,-2a,h_0-a\}\}
\\ &\iff
b\in\{\pm a,h_0\pm a\}
\\ &\qquad\quad
\text{or } h_0=2a\text{ and } \{b,-b\}
\in\{\{a,3a\},\{-a,a\}\}
\\ &\iff
b\in\{\pm a,h_0\pm a\}.
\end{align*}
\end{proof}

\begin{lem}\label{lem:R2}
Let $a\in A\setminus\Omega_1(A)$, $b\in A\setminus\{0\}$
and $h\in\Omega_1(A)\setminus\{0,b\}$. Then
\begin{enumerate}
\item
$\orbt{a,-a}=\orbt{b,h}$
if and only if
$a=\pm b$ and $2a=h$. In particular, $\orbt{a,-a}\in\cT_2$
if and only if $a\in\Omega_2(A)$.
\item
$\orbt{a,-a,h_0}=\orbt{b,h,h+b}$
if and only if
$a=\pm b$ and $2a=h_0=h$.
\end{enumerate}
\end{lem}
\begin{proof}
(i) By (\ref{eq:R1}),
\begin{align*}
\orbt{a,-a}=\orbt{b,h}
&\iff
\{b,h\}\in\{\{a,-a\},\{a,2a\},\{-a,-2a\}\}
\\ &\iff
(b,h)\in\{(a,2a),(-a,-2a)\}
\\ &\iff
a=\pm b\text{ and }2a=h.
\end{align*}
Thus
\begin{align*}
\orbt{a,-a}\in\cT_2
&\iff
\exists b\in A\setminus\{0\},\;
\exists h\in\Omega_1(A)\setminus\{0,b\}
\\ &\qquad\quad\text{ s.t. }
a=\pm b\text{ and }h=2a
\\ &\iff
a\in\Omega_2(A)\text{ and }
\exists b\in\{\pm a\}\setminus\{0,2a\}
\\ &\iff
a\in\Omega_2(A),
\end{align*}
since $\{\pm a\}\cap\{0,2a\}=\emptyset$ when
$a\in\Omega_2(A)\setminus\Omega_1(A)$.

(ii) By (\ref{eq:B0b}),
\begin{align*}
&\orbt{a,-a,h_0}=\orbt{b,h,h+b}
\\ &\iff
\{b,h,h+b\}\in\{\{a,-a,h_0\},\{h_0+a,h_0-a,h_0\},
\\ &\qquad\qquad\qquad\qquad\qquad
\{a,2a,h_0+a\},\{-a,-2a,h_0-a\}\}
\\ &\iff
h_0=h\text{ and }\{b,h_0+b\}\in\{\{a,-a\},\{h_0+a,h_0-a\}\}
\\ &\qquad\quad\text{or }
2a=h\text{ and } \{b,h+b\}\in\{\pm\{a,h_0+a\}\}
\\ &\iff
h_0=h,\;b\in\{\pm a,h_0\pm a\}\text{ and }2a=h_0,
\\ &\qquad\quad\text{or }
2a=h,\;b\in\{\pm a,h_0\pm a\}\text{ and }h_0=h
\\ &\iff
a=\pm b\text{ and } 2a=h_0=h.
\end{align*}
\end{proof}

\begin{lem}\label{lem:R3}
Let $a,a'\in A\setminus\{0\}$, $h\in\Omega_1(A)\setminus\{0,a\}$,
and $h'\in\Omega_1(A)\setminus\{0,a'\}$.
Then the following are equivalent:
\begin{enumerate}
\item
$\orbt{a,h}=\orbt{a',h'}$,
\item
$\orbt{a,h,h+a}=\orbt{a',h',h'+a'}$,
\item 
\begin{equation}\label{eq:T2_1}
a,a'\in\Omega_1(A)
\text{ and }\langle a,h\rangle=\langle a',h'\rangle,
\end{equation}
or
\begin{equation}\label{eq:T2_2}
a,a'\notin\Omega_1(A),\;
h=h'
\text{ and }
a'\in\{\pm a,h\pm a\}.
\end{equation}
\end{enumerate}
\end{lem}
\begin{proof}
Suppose $a\in\Omega_1(A)$. Then 
$\orbt{a,h}=\orbt{a',h'}$ or $\orbt{a,h,h+a}=\orbt{a',h',h'+a'}$
implies $a'\in\Omega_1(A)$ by (\ref{eq:SS}) or (\ref{eq:B0a}).
If $a,a'\in\Omega_1(A)$, then
\begin{align*}
&\orbt{a,h}=\orbt{a',h'}
\\ &\iff
\{a',h'\}\in\{\{a,h\},\{h,h+a\},\{a,h+a\}\}
&&\text{(by (\ref{eq:SS}))}
\\ &\iff
\langle a,h\rangle=\langle a',h'\rangle
\\ &\iff
\{a,h,h+a\}=\{a',h',h'+a'\}
\\ &\iff
\orbt{a,h,h+a}=\orbt{a',h',h'+a'}
&&\text{(by (\ref{eq:B0a})).}
\end{align*}
Now suppose $a,a'\not\in\Omega_1(A)$. Then
\begin{align*}
&\orbt{a,h}=\orbt{a',h'}
\\ &\iff
\{a',h'\}\in\{\pm\{a,h\},\pm\{h,h+a\},\pm\{a,h+a\}\}
&&\text{(by (\ref{eq:SS}))}
\\ &\iff
\{a',h'\}\in\{\pm\{a,h\},\pm\{h,h+a\}\}
\\ &\iff
h=h'\text{ and } a'\in\{\pm a,h\pm a\}
\\ &\iff
h=h'\text{ and } \{a',h'+a'\}=\pm\{a,h+a\}
\\ &\iff
\{a',h',h'+a'\}=\pm\{a,h,h+a\}
\\ &\iff
\orbt{a,h,h+a}=\orbt{a',h',h'+a'}
&&\text{(by (\ref{eq:B0a})).}
\end{align*}
\end{proof}

\begin{lem}\label{lem:2T1}
Let $a\in A\setminus\Omega_1(A)$, $T\in\orbt{a,-a}$
and $B\in\orbt{a,-a,h_0}$. Then
\begin{align}
&|\{a'\in A\setminus\Omega_2(A)\mid T\in\orbt{a',-a'}\}|
=2\quad\text{if $a\notin\Omega_2(A)$,}\label{eq:NT3}\\
&|\{a'\in A\setminus\Omega_1(A)\mid 2a'=h_0,\;
B\in\orbt{a',-a',h_0}\}|
=2\quad\text{if $2a=h_0$,}\label{eq:NB1}\\
&|\{a'\in A\setminus\Omega_1(A)\mid 2a'\neq h_0,\;
B\in\orbt{a',-a',h_0}\}|
=4\quad\text{if $2a\neq h_0$.}\label{eq:NB2}
\end{align}
\end{lem}
\begin{proof}
If $a\notin\Omega_2(A)$, then by
Lemma~\ref{lem:R1}(i), we have
\begin{align*}
\{a'\in A\setminus\Omega_2(A)\mid T\in\orbt{a',-a'}\}
&=\{a'\in A\setminus\Omega_2(A)\mid \orbt{a,-a}=\orbt{a',-a'}\}
\nexteq\{\pm a\}.
\end{align*}
Thus
\[
|\{a'\in A\setminus\Omega_2(A)\mid T\in\orbt{a',-a'}\}|
=|\{\pm a\}|=2.
\]

Suppose $a\in A\setminus\Omega_1(A)$.
By Lemma~\ref{lem:R1}(ii), we have
\begin{align*}
&\{a'\in A\setminus\Omega_1(A)\mid B\in\orbt{a',-a',h_0}\}
\\ &=
\{a'\in A\setminus\Omega_1(A)\mid \orbt{a,-a,h_0}
=\orbt{a',-a',h_0}\}
\nexteq
\{\pm a,h_0\pm a\}
\\ &\subset
\begin{cases}
\{a'\in A\setminus\Omega_1(A)\mid 2a'=h_0\}
&\text{if $2a=h_0$,}\\
\{a'\in A\setminus\Omega_1(A)\mid 2a'\neq h_0\}
&\text{otherwise.}
\end{cases}
\end{align*}
Thus
\[
\{a'\in A\setminus\Omega_1(A)\mid 2a'=h_0,\;
B\in\orbt{a',-a',h_0}\}
=|\{\pm a\}|
=2
\]
if $2a=h_0$, and
\[
\{a'\in A\setminus\Omega_1(A)\mid 2a'\neq h_0,\;
B\in\orbt{a',-a',h_0}\}
=|\{\pm a,h_0\pm a\}|=4
\]
if $2a\neq h_0$.
\end{proof}

\begin{lem}\label{lem:2T2}
Let $a\in A\setminus\Omega_1(A)$, $h\in\Omega_1(A)\setminus\{0\}$,
$T\in\orbt{a,h}$
and $B\in\orbt{a,h,h+a}$. Then
\begin{align}
&|\{a'\in A\setminus\Omega_1(A)\mid 2a'=h,\;T\in\orbt{a',h}\}|
=2\quad\text{if $2a=h$,}\label{eq:2.5T1}\\
&|\{a'\in A\setminus\Omega_1(A)\mid 2a'\neq h,\;T\in\orbt{a',h}\}|
=4\quad\text{if $2a\neq h$,}\label{eq:2.5T2}\\
&|\{a'\in A\setminus\Omega_1(A)\mid 2a'=h,\;B\in\orbt{a',h,h+a'}\}|
=2\quad\text{if $2a=h$,}\label{eq:2.5B1}\\
&|\{a'\in A\setminus\Omega_1(A)\mid 2a'\neq h,\;B\in\orbt{a',h,h+a'}\}|
=4\quad\text{if $2a\neq h$.}\label{eq:2.5B2}
\end{align}
\end{lem}
\begin{proof}
By Lemma~\ref{lem:R3}, we have
\begin{align*}
&\{a'\in A\setminus\Omega_1(A)\mid T\in\orbt{a',h}\}
\\ &=\{a'\in A\setminus\Omega_1(A)\mid \orbt{a,h}=\orbt{a',h}\}
\nexteq
\{\pm a,h\pm a\}
\\ &\subset
\begin{cases}
\{a'\in A\setminus\Omega_1(A)\mid 2a'=h\}&\text{if $2a=h$,}\\
\{a'\in A\setminus\Omega_1(A)\mid 2a'\neq h\}&\text{otherwise.}
\end{cases}
\end{align*}
Thus
\[
|\{a'\in A\setminus\Omega_1(A)\mid 2a'=h,\;T\in\orbt{a',h}\}|
=|\{\pm a\}|=2
\]
if $2a=h$, and
\[
|\{a'\in A\setminus\Omega_1(A)\mid 2a'\neq h,\;T\in\orbt{a',h}\}|
=|\{\pm a,h\pm a\}|=4
\]
if $2a\neq h$.

Also, by Lemma~\ref{lem:R3}, we have
\begin{align*}
&\{a'\in A\setminus\Omega_1(A)\mid B\in\orbt{a',h,h+a'}\}
\\ &=\{a'\in A\setminus\Omega_1(A)\mid \orbt{a,h,h+a}=\orbt{a',h,h+a'}\}
\nexteq
\{\pm a,h\pm a\}
\\ &\subset
\begin{cases}
\{a'\in A\setminus\Omega_1(A)\mid 2a'=h\}&\text{if $2a=h$,}\\
\{a'\in A\setminus\Omega_1(A)\mid 2a'\neq h\}&\text{otherwise.}
\end{cases}
\end{align*}
Thus
\[
|\{a'\in A\setminus\Omega_1(A)\mid 2a'=h,\;B\in\orbt{a',h,h+a'}\}|
=|\{\pm a\}|=2
\]
if $2a=h$, and
\[
|\{a'\in A\setminus\Omega_1(A)\mid 2a'\neq h,\;B\in\orbt{a',h,h+a'}\}|
=|\{\pm a,h\pm a\}|=4
\]
if $2a\neq h$.
\end{proof}

\begin{lem}\label{lem:NOTQ}
\begin{enumerate}
\item 
Let $a\in A\setminus\Omega_1(A)$.
Then 
\begin{align}\label{eq:NQ1}
|\orbt{a,-a}|&=v,
\\ \intertext{and}
\label{eq:NQ2}
|\orbt{a,-a,h_0}|&=\begin{cases}
\frac{v}{4}&\text{if $2a=h_0$,}\\
v&\text{otherwise.}
\end{cases}
\end{align}
\item Let $a\in A\setminus\{0\}$, $h\in\Omega_1(A)\setminus
\{0\}$, and $a\neq h$. Then
\begin{align}\label{eq:NQ3}
|[a,h]|&=\begin{cases}
v &\text{if $2a=h$ or $a\in\Omega_1(A)$,}\\
2v&\text{otherwise,}
\end{cases}
\\ \intertext{and}
\label{eq:NQ4}
|\orbt{a,h,h+a}|&=\begin{cases}
\frac{v}{2}&\text{if $a\notin\Omega_1(A)$ and $2a\neq h$,}\\
\frac{v}{4}&\text{if $a\in\Omega_1(A)$.}
\end{cases}
\end{align}
\end{enumerate}
\end{lem}
\begin{proof}
If $a_1,\dots,a_t\in A\setminus\{0\}$ are distinct, then
counting the number of pairs 
$(x,T)$ with $x\in T\in\orbt{a_1,\dots,a_t}$, we find
\begin{align}
(t+1)|\orbt{a_1,\dots,a_t}|&=
\sum_{x\in A}|\{T\in\orbt{a_1,\dots,a_t}\mid x\in T\}|
\nnexteq
\sum_{x\in A}|\{T\in\orbt{a_1,\dots,a_t}\mid 0\in T\}|
\nnexteq
v|\{T\in\binom{A}{t}\mid \{0\}\cup T\in\orbt{a_1,\dots,a_t}\}|.
\label{eq:S}
\end{align}
Thus, for $a\in A\setminus\Omega_1(A)$, (\ref{eq:R1}) implies
\begin{equation}\label{eq:N1}
|\orbt{a,-a}|
=\frac{v}{3}|\{\{a,-a\},\{a,2a\},\{-a,-2a\}\}|,
\end{equation}
and (\ref{eq:B0b}) implies
\begin{align}
|\orbt{a,-a,h_0}|
&=\frac{v}{4}|\{\{a,-a,h_0\},\{h_0,h_0+a,h_0-a\},
\nonumber\\ &\qquad\quad
\{a,2a,h_0+a\},\{-a,-2a,h_0-a\}\}|.
\label{eq:13}
\end{align}
Since $v\not\equiv0\pmod3$, we have $3a\neq0$ for any
$a\in A\setminus\{0\}$.
Thus (\ref{eq:NQ1}) follows from (\ref{eq:N1}), while
(\ref{eq:NQ2}) follows from (\ref{eq:13}). 

As for (ii),
if $a\in A\setminus\{0\}$ and $h\in\Omega_1(A)\setminus\{0\}$
and $a\neq h$, then (\ref{eq:SS}) and (\ref{eq:S}) imply
\begin{equation}\label{eq:N2}
|\orbt{a,h}|=
\frac{v}{3}|\{\pm\{a,h\},\pm\{h,h+a\},\pm\{a,h+a\}\}|,
\end{equation}
while (\ref{eq:B0a}) and (\ref{eq:S}) imply
\begin{equation}\label{eq:14}
|\orbt{a,h,h+a}|=
\frac{v}{4}|\{\{a,h,h+a\},\{-a,h,h-a\}\}|.
\end{equation}
By (\ref{eq:N2}), we have
\begin{align*}
|\orbt{a,h}|
&=\begin{cases}
\frac{v}{3}|\{\pm\{a,2a\},\pm\{2a,3a\},\pm\{a,3a\}\}|
&\text{if $2a=h$,}\\
\frac{v}{3}|\{\{a,h\},\{h,h+a\},\{a,h+a\}\}|
&\text{if $a\in\Omega_1(A)$,}\\
\frac{v}{3}|\{\pm\{a,h\},\pm\{h,h+a\},\pm\{a,h+a\}\}|
&\text{otherwise,}
\end{cases}
\nexteq
\begin{cases}
\frac{v}{3}|\{\{a,2a\},\{-a,-2a\},\{a,-a\}\}|
&\text{if $2a=h$,}\\
v&\text{if $a\in\Omega_1(A)$,}\\
2v&\text{otherwise,}
\end{cases}
\nexteq
\begin{cases}
v&\text{if $2a=h$ or $a\in\Omega_1(A)$,}\\
2v&\text{otherwise.}
\end{cases}
\end{align*}
This proves (\ref{eq:NQ3}). Finally, (\ref{eq:NQ4})
follows from (\ref{eq:14}).
\end{proof}

We now compute the number of triples $T$ satisfying
$\orb{\hA}{T}\in\cT_1\cup\cT_2$. 

\begin{lem}\label{lem:NT}
We have
\[
|\{T\mid\orb{\hA}{T}\in\cT_1\cup\cT_2\}|
=
\frac12 v^2\omega_1-\frac16 v(2\omega_1^2+3\omega_2-2).
\]
\end{lem}
\begin{proof}
Counting the number of pairs $(a,T)$ with $a\in A\setminus\Omega_1(A)$,
$h \in \Omega_1(A)\setminus\{0\}$,
$2a=h$ and
$T\in\orbt{a,h}$ using (\ref{eq:2.5T1}), we find
\[
2\left|\bigcup_{\substack{a\in A\setminus\Omega_1(A)\\ 2a=h}}
\orbt{a,h}\right|
=\sum_{\substack{a\in A\setminus\Omega_1(A)\\ 2a=h}}
|\orbt{a,h}|.
\]
Thus, by (\ref{eq:NQ3}), we have
\begin{equation}\label{eq:2.7a}
\left|\bigcup_{\substack{a\in A\setminus\Omega_1(A)\\ 2a=h}}
\orbt{a,h}\right|
=\frac{v}{2}|\{a\in A\setminus\Omega_1(A)\mid 2a=h\}|.
\end{equation}
Counting the number of pairs $(a,T)$ with
$a\in A\setminus\Omega_1(A)$, $2a\neq h$ and
$T\in\orbt{a,h}$ using (\ref{eq:2.5T2}), we find
\[
4\left|\bigcup_{\substack{a\in A\setminus\Omega_1(A)\\ 2a\neq h}}
\orbt{a,h}\right|
=\sum_{\substack{a\in A\setminus\Omega_1(A)\\ 2a\neq h}}
|\orbt{a,h}|.
\]
Thus, by (\ref{eq:NQ3}), we have
\begin{equation}\label{eq:2.7b}
\left|\bigcup_{\substack{a\in A\setminus\Omega_1(A)\\ 2a\neq h}}
\orbt{a,h}\right|
=\frac{v}{2}|\{a\in A\setminus\Omega_1(A)\mid 2a\neq h\}|.
\end{equation}
Therefore,
\begin{align}
&\left|
\bigcup_{h\in\Omega_1(A)\setminus\{0\} }
\bigcup_{a\in A\setminus\Omega_1(A) }
\orbt{a,h}
\right|
\nonumber\\ &=
\sum_{h\in\Omega_1(A)\setminus\{0\} }
\left|
\bigcup_{\substack{a\in A\setminus\Omega_1(A) \\ 2a=h}}
\orbt{a,h}
\right|
\nonumber\\ &\quad+
\sum_{h\in\Omega_1(A)\setminus\{0\} }
\left|
\bigcup_{\substack{a\in A\setminus\Omega_1(A) \\ 2a\neq h}}
\orbt{a,h}
\right|
&&\text{(by (\ref{eq:T2_2}))}
\nnexteq
\sum_{h\in\Omega_1(A)\setminus\{0\} }
\frac{v}{2}|A\setminus \Omega_1(A)|
&&\text{(by (\ref{eq:2.7a}), (\ref{eq:2.7b}))}
\nnexteq
\frac12 v(v-\omega_1)(\omega_1-1).
\label{eq:8B}
\end{align}

If $H\in\qbinom{\Omega_1(A)}{2}$ and $H=\langle h',h''\rangle$,
then
\begin{equation}\label{eq:8B1}
\left|\bigcup_{\{a,h\}\in\binom{H\setminus\{0\}}{2}}\orbt{a,h}\right|
=|\orbt{h',h''}|
=v
\end{equation}
by (\ref{eq:T2_1}) and (\ref{eq:NQ3}).
Thus we have
\begin{align}
\left|
\bigcup_{h\in\Omega_1(A)\setminus\{0\}}
\bigcup_{a\in\Omega_1(A)\setminus\langle h\rangle}
\orbt{a,h}
\right|
&=
\left|
\bigcup_{H\in\qbinom{\Omega_1(A)}{2}}
\bigcup_{\{a,h\}\in\binom{H\setminus\{0\}}{2}}
\orbt{a,h}
\right|
\nnexteq
\sum_{H\in\qbinom{\Omega_1(A)}{2}}
\left|
\bigcup_{\{a,h\}\in\binom{H\setminus\{0\}}{2}}
\orbt{a,h}
\right|
&&\text{(by (\ref{eq:T2_1}))}
\nnexteq
\sum_{H\in\qbinom{\Omega_1(A)}{2}}
v
&&\text{(by (\ref{eq:8B1}))}
\nnexteq 
v\left|\qbinom{\Omega_1(A)}{2}\right|
\nnexteq 
\frac16 v(\omega_1-1)(\omega_1-2)
&&\text{(by (\ref{eq:qbin})).}
\label{eq:8B2}
\end{align}
Therefore
\begin{align}
&|\{T\mid\orb{\hA}{T}\in\cT_2\}|
\nonumber\\ &=
\left|
\bigcup_{\substack{T\in\binom{A}{3} \\
\orb{\hA}{T}\in\cT_2}}
\orb{\hA}{T}
\right|
\nnexteq
\left|
\bigcup_{h\in\Omega_1(A)\setminus\{0\}}
\bigcup_{a\in A\setminus\langle h\rangle}
\orbt{a,h}
\right|
\nnexteq
\left|
\bigcup_{h\in\Omega_1(A)\setminus\{0\}}
\bigcup_{a\in A\setminus\Omega_1(A)}
\orbt{a,h}
\right|
\nonumber\\ &\quad
+\left|
\bigcup_{h\in\Omega_1(A)\setminus\{0\}}
\bigcup_{a\in\Omega_1(A)\setminus\langle h\rangle}
\orbt{a,h}
\right|
&&\text{(by Lemma~\ref{lem:R3})}
\nnexteq
\frac12 v(v-\omega_1)(\omega_1-1)
+
\frac16 v(\omega_1-1)(\omega_1-2)
&&\text{(by (\ref{eq:8B}), (\ref{eq:8B2}))}
\nnexteq 
\frac16 v(\omega_1-1)(3v-2\omega_1-2).
\label{eq:R7}
\end{align}

On the other hand, counting the number of pairs
$(a,T)$ with $a\in A\setminus\Omega_2(A)$ and
$T\in\orbt{a,-a}$ using (\ref{eq:NT3}), we find
\begin{equation}\label{eq:2.7c}
2\left|
\bigcup_{a\in A\setminus\Omega_2(A)}
\orbt{a,-a}
\right|
=\sum_{a\in A\setminus\Omega_2(A)}
|\orbt{a,-a}|.
\end{equation}
Thus we have 
\begin{align*}
&|\{T\mid\orb{\hA}{T}\in\cT_1\setminus\cT_2\}|
\\ &=
\left|
\bigcup_{\substack{T\in\binom{A}{3}\\
\orb{\hA}{T}\in\cT_1\setminus\cT_2}}
\orb{\hA}{T}
\right|
\nexteq
\left|
\bigcup_{\substack{a\in A\setminus\Omega_1(A)\\
\orbt{a,-a}\notin\cT_2}}
\orbt{a,-a}
\right|
\nexteq
\left|
\bigcup_{a\in A\setminus\Omega_2(A)}
\orbt{a,-a}
\right|
&&\text{(by Lemma~\ref{lem:R2}(i))}
\nexteq
\frac12\sum_{a\in A\setminus\Omega_2(A)}
|\orbt{a,-a}|
&&\text{(by (\ref{eq:2.7c}))}
\nexteq
\frac{1}{2}v
|A\setminus\Omega_2(A)|
&&\text{(by (\ref{eq:NQ1}))}
\nexteq
\frac12 v(v-\omega_2).
\end{align*}
Combining this with (\ref{eq:R7}), we find
\begin{align*}
|\{T\mid\orb{\hA}{T}\in\cT_1\cup\cT_2\}|
&=\frac16 v(\omega_1-1)(3v-2\omega_1-2)
+\frac12 v(v-\omega_2)
\nexteq
\frac12 v^2\omega_1-\frac16 v(2\omega_1^2+3\omega_2-2).
\end{align*}
\end{proof}

Next we compute the number of quadruples $B$
satisfying $\orb{\hA}{B}\in\cQ_1\cup\cQ_2\cup\cQ_3$.
Let 
\begin{align}
\cB_0&=\{B\mid\orb{\hA}{B}\in\cQ_1\cup\cQ_2\cup\cQ_3\},\label{eq:B0}\\
A_0&=\{a\in A\mid 2a=h_0\},\nonumber\\
\omega_0&=|A_0|.\nonumber
\end{align}

\begin{lem}\label{lem:NQ}
We have
\[
|\cB_0|=
\frac18 v^2\omega_1
-\frac{1}{24}v(2\omega_1^2+3\omega_2-2).
\]
\end{lem}
\begin{proof}
Counting the number of pairs $(a,B)$ with
$a\in A\setminus\Omega_1(A)$, $2a=h_0$ and
$B\in\orbt{a,-a,h_0}$ using (\ref{eq:NB1}), we find
\[
2\left|\bigcup_{\substack{a\in A\setminus\Omega_1(A) \\ 
2a=h_0 } }
\orbt{a,-a,h_0}\right|
=\sum_{\substack{a\in A\setminus\Omega_1(A) \\ 2a=h_0 }}
|\orbt{a,-a,h_0}|.
\]
Thus, by (\ref{eq:NQ2}), we have
\begin{align}
\left|\bigcup_{\substack{a\in A\setminus\Omega_1(A) \\ 
2a=h_0 } }
\orbt{a,-a,h_0}\right|
&=
\frac{v}{8}
|\{a\in A\setminus\Omega_1(A)\mid 2a=h_0\}|
\nnexteq
\frac18v\omega_0.
\label{eq:B1A}
\end{align}
Counting the number of pairs $(a,B)$ with
$a\in A\setminus\Omega_1(A)$, $2a\neq h_0$ and
$B\in\orbt{a,-a,h_0}$ using (\ref{eq:NB2}), we find
\[
4\left|\bigcup_{\substack{a\in A\setminus\Omega_1(A) \\ 
2a\neq h_0 } }
\orbt{a,-a,h_0}\right|
=\sum_{\substack{a\in A\setminus\Omega_1(A) \\ 2a\neq h_0 }}
|\orbt{a,-a,h_0}|.
\]
Thus, by (\ref{eq:NQ2}), we have
\begin{align}
\left|\bigcup_{\substack{a\in A\setminus\Omega_1(A) \\ 
2a\neq h_0 } }
\orbt{a,-a,h_0}\right|
&=
\frac{v}{4}
|\{a\in A\setminus\Omega_1(A)\mid 2a\neq h_0\}|
\nnexteq
\frac14v(v-\omega_1-\omega_0).
\label{eq:B1B}
\end{align}
Therefore
\begin{align}
&|\{B\mid \orb{\hA}{B}\in\cQ_1\}|
\nonumber\\ &=
\left|\bigcup_{a\in A\setminus\Omega_1(A)}
\orbt{a,-a,h_0}\right|
\nnexteq
\left|\bigcup_{\substack{a\in A\setminus\Omega_1(A) \\ 
2a=h_0 } }
\orbt{a,-a,h_0}\right|
\nonumber\\ &\quad
+\left|\bigcup_{\substack{a\in A\setminus\Omega_1(A) \\ 
2a \ne h_0 } }
\orbt{a,-a,h_0}\right|
&&\text{(by Lemma~\ref{lem:R1}(ii))}
\nnexteq
\frac{1}{8}vw_0+\frac{1}{4}v(v-\omega_1-\omega_0)
&&\text{(by (\ref{eq:B1A}) and (\ref{eq:B1B}))}
\nnexteq
\frac{1}{8}v(2v-2\omega_1-\omega_0).
\label{eq:B1}
\end{align}

Let $h\in\Omega_1(A)\setminus\langle h_0\rangle$.
Counting the number of pairs $(a,B)$ with
$a\in A\setminus\Omega_1(A)$, $2a \ne h$ and $B\in
\orbt{a,h,h+a}$ using (\ref{eq:2.5B2}), we find
\[
4\left|\bigcup_{\substack{a\in A\setminus\Omega_1(A) \\
2a \ne h }}
\orbt{a,h,h+a}\right|
=\sum_{\substack{a\in A\setminus\Omega_1(A) \\
2a \ne h }}
|\orbt{a,h,h+a}|.
\]
Thus, by (\ref{eq:NQ4}), we have
\begin{equation}\label{eq:2.8a}
\left|\bigcup_{\substack{a\in A\setminus\Omega_1(A) \\
2a \ne h }}
\orbt{a,h,h+a}\right|
=\frac{v}{8}|\{a\in A\setminus\Omega_1(A) \mid
2a \ne h \}|.
\end{equation}
Therefore
\begin{align}
&|\{B\mid\orb{\hA}{B}\in\cQ_2\}|
\nonumber\\ &=
\left|\bigcup_{\substack{a\in A\setminus\Omega_1(A) \\
h\in\Omega_1(A)\setminus\langle h_0\rangle \\
2a \ne h }}
\orbt{a,h,a+h}\right|
\nnexteq
\sum_{h\in\Omega_1(A)\setminus\langle h_0\rangle}
\left|\bigcup_{\substack{a\in A\setminus\Omega_1(A) \\
2a \ne h }}
\orbt{a,h,a+h}\right|
&&\text{(by (\ref{eq:T2_2}))}
\nnexteq
\frac{v}{8}
\sum_{h\in\Omega_1(A)\setminus\langle h_0\rangle}
|\{a\in A\setminus\Omega_1(A)\mid 2a\neq h\}|
&&\text{(by (\ref{eq:2.8a}))}
\nnexteq
\frac{v}{8}
\sum_{a\in A\setminus\Omega_1(A)}
|(\Omega_1(A)\setminus\langle h_0\rangle)\setminus\{2a\}|
\nnexteq
\frac{v}{8}\left(
\sum_{a\in (A\setminus\Omega_2(A))\cup A_0}(\omega_1-2)
+\sum_{a\in \Omega_2(A)\setminus(A_0\cup\Omega_1(A))}
(\omega_1-3)\right)
\nnexteq
\frac{v}{8}
\left(
(v-\omega_2+\omega_0)(\omega_1-2)
+(\omega_2-\omega_0-\omega_1)(\omega_1-3)\right)
\nnexteq
\label{eq:N2x}
\frac{v}{8}((\omega_1-2)v-\omega_1^2+\omega_0+3\omega_1-\omega_2).
\end{align}

If $H\in\qbinom{\Omega_1(A)}{2}$ and $H=\langle h',h''\rangle$,
then 
\begin{equation}\label{eq:hh}
\left|\bigcup_{\{h_1,h_2\}\in\binom{H\setminus\{0\}}{2}}
\orbt{h_1,h_2,h_1+h_2}\right|
=|\orbt{h',h'',h'+h''}|=\frac{v}{4}
\end{equation}
by (\ref{eq:T2_1}) and Lemma~\ref{lem:NOTQ}(ii). Thus we have
\begin{align}
&|\{B\mid\orb{\hA}{B}\in\cQ_3\}|
\nonumber\\ &=\left|
\bigcup_{\substack{ h,h'\in\Omega_1(A)\setminus\{0\} \\
h\ne h'} }
\orbt{h,h',h+h'}
\right|
\nnexteq
\left|
\bigcup_{H\in\qbinom{\Omega_1(A)}{2}}
\bigcup_{\{h,h'\}\in\binom{H\setminus\{0\}}{2} }
\orbt{h,h',h+h'}
\right|
\nnexteq
\sum_{H\in\qbinom{\Omega_1(A)}{2}}
\left|
\bigcup_{\{h,h'\}\in\binom{H\setminus\{0\}}{2} }
\orbt{h,h',h+h'}
\right|
&&\text{(by (\ref{eq:T2_1}))}
\nnexteq
\sum_{H\in\qbinom{\Omega_1(A)}{2}}
\frac{v}{4}
&&\text{(by (\ref{eq:hh}))}
\nnexteq
\frac{v}{4}
\left|\qbinom{\Omega_1(A)}{2}\right|
\nnexteq
\frac{1}{24} v(\omega_1-1)(\omega_1-2)
&&\text{(by (\ref{eq:qbin})).}
\label{eq:N3}
\end{align}
Since the sets $\cQ_1,\cQ_2$ and $\cQ_3$ are pairwise disjoint
by Lemma~\ref{lem:R2}(ii) and Lemma~\ref{lem:R3}, the equations
(\ref{eq:B1}), (\ref{eq:N2x}) and (\ref{eq:N3}) can be
combined to give
\begin{align*}
|\cB_0|&=
|\{B\mid\orb{\hA}{B}\in\cQ_1\}|
+|\{B\mid\orb{\hA}{B}\in\cQ_2\}|
\\ &\quad
+|\{B\mid\orb{\hA}{B}\in\cQ_3\}|
\nexteq
\frac18 v(2v-2\omega_1-\omega_0)
+\frac18 v((\omega_1-2)v-\omega_1^2-\omega_2+3\omega_1+\omega_0)
\\ &\quad+
\frac{1}{24} v(\omega_1-1)(\omega_1-2)
\nexteq
\frac18 v^2\omega_1
-\frac{1}{24}v(2\omega_1^2+3\omega_2-2).
\end{align*}
\end{proof}

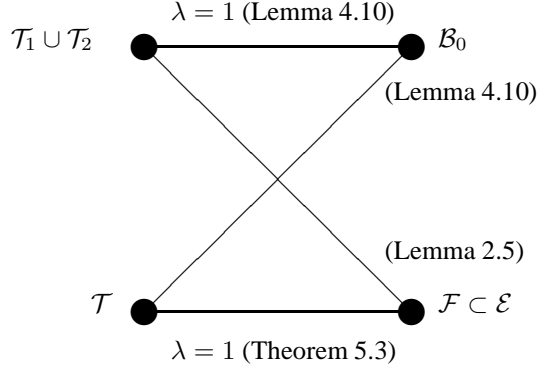
\begin{center}
\begin{picture}(140,140)
\put(20,20){
\put(0,0){\circle*{10}}
\put(0,100){\circle*{10}}
\put(100,0){\circle*{10}}
\put(100,100){\circle*{10}}
\thinlines
\put(0,0){\line(1,1){100}}
\put(0,100){\line(1,-1){100}}
\thicklines
\put(0,0){\line(1,0){100}}
\put(0,100){\line(1,0){100}}
\put(-20,0){$\mathcal{T}$}
\put(-50,100){$\mathcal{T}_1\cup \mathcal{T}_2$}
\put(10,-18){$\lambda=1$ (Theorem~\ref{thm:Kohler})}
\put(10,110){$\lambda=1$ (Lemma~\ref{lem:ps1})}
\put(90,80){(Lemma~\ref{lem:ps1})}
\put(90,20){(Lemma~\ref{lem:blks})}
\put(110,100){$\mathcal{B}_0$}
\put(110,0){$\mathcal{F}\subset\mathcal{E}$}
}
\end{picture}
\end{center}

\begin{lem}\label{lem:ps1}
If $B\in\cB_0$, $T\subset B$ and $|T|=3$, then
$\orb{\hA}{T}\in\cT_1\cup\cT_2$.
Conversely, if $\orb{\hA}{T}\in\cT_1\cup\cT_2$, then
there exists a unique $B\in\cB_0$ such that $T\subset B$.
\end{lem}
\begin{proof}
Suppose $B\in\cB_0$. 

If $\orb{\hA}{B}\in\cQ_1$, then we may
assume without loss of generality that $B=\{0,a,-a,h_0\}$
for some $a\in A\setminus\Omega_1(A)$. If $T\in
\{\{0,a,-a\},\{a,-a,h_0\}=\{0,a+h_0,-(a+h_0)\}+h_0\}$,
then $\orb{\hA}{T}\in\cT_1$. If $T=\pm \{0,a,h_0\}$, then
$\orb{\hA}{T}\in\cT_2$.

If $\orb{\hA}{B}\in\cQ_2$, then we may
assume without loss of generality that $B=\{0,a,h,h+a\}$
for some $a\in A\setminus\Omega_1(A)$ and $h\in\Omega_1(A)
\setminus\langle h_0\rangle$ with $2a\neq h$.
Then for $T\in\{\{0,a,h\}, \{0,h+a,h\}, 
\{0,a,h+a\}=\{0,-a,h\}+a,
\{a,h,h+a\}=\{0,h-a,h\}+a\}$, we have
$\orb{\hA}{T}\in\cT_2$. 

If $\orb{\hA}{B}\in\cQ_3$, then we may
assume without loss of generality that $B=\{0,h,h',h+h'\}$
for some $h,h'\in \Omega_1(A)\setminus\{0\}$ with $h\neq h'$.
Then for $T\in\{\{0,h,h'\}, \{0,h,h+h'\}, 
\{0,h',h+h'\},
\{h,h',h+h'\}=\{0,h,h'\}+h+h'\}$, we have
$\orb{\hA}{T}\in\cT_2$. 

Conversely, suppose $\orb{\hA}{T}\in\cT_1\cup\cT_2$. 

If $\orb{\hA}{T}\in\cT_1$, then we may assume without loss
of generality that $T=\{0,a,-a\}$ for some $a\in A\setminus
\Omega_1(A)$. Then $T\subset B=\{0,a,-a,h_0\}$
and $\orb{\hA}{B}\in\cQ_1$, so $B\in\cB_0$.

If $\orb{\hA}{T}\in\cT_2$, then we may assume without loss
of generality that $T=\{0,a,h\}$ for some $a\in A$ and $h
\in\Omega_1(A)\setminus\{0,a\}$.
If $a\in\Omega_1(A)$, then
$T\subset B=\{0,a,h,h+a\}$
and $\orb{\hA}{B}\in\cQ_3$, so $B\in\cB_0$.
If $a\notin\Omega_1(A)$ and $h=h_0$, 
then $T\subset B=\{0,a,-a,h_0\}$
and $\orb{\hA}{B}\in\cQ_1$, so $B\in\cB_0$.
Suppose $a\notin\Omega_1(A)$ and $h\neq h_0$.
If $2a\neq h$, then
$T\subset B=\{0,a,h,h+a\}$
and $\orb{\hA}{B}\in\cQ_2$, so $B\in\cB_0$.
If $2a=h$, then $T\subset B=\{0,a,h,h_0+a\}=
\{0,a,-a,h_0\}+a$ and
$\orb{\hA}{B}\in\cQ_1$, so $B\in\cB_0$.

Therefore, we have shown that 
there exists at least one
$B\in\cB_0$ such that $T\subset B$.
Counting the number of pairs $(T,B)$ with $T\in\binom{A}{3}$,
$\orb{\hA}{T}\in\cT_1\cup\cT_2$ and $T\subset B\in\cB_0$
using the first part, we find
\begin{align*}
4|\cB_0|
&=\sum_{\substack{
T\in\binom{A}{3}\\ \orb{\hA}{T}\in\cT_1\cup\cT_2}}
|\{B\in\cB_0 \mid T\subset B\}|
\\ &\geq
|\{T\mid\orb{\hA}{T}\in\cT_1\cup\cT_2\}|.
\end{align*}
By Lemma~\ref{lem:NT} and Lemma~\ref{lem:NQ}, we
conclude $|\{B\in\cB_0\mid T\subset B\}|=1$ whenever
$\orb{\hA}{T}\in\cT_1\cup\cT_2$.
\end{proof}

\section{Reversible Steiner quadruple systems}

In this section, we let
$A$ be an abelian group of order $v\equiv2$ or $4\pmod6$, 
and use the notation
introduced in Section~2--4. Specifically, we continue to use
notation introduced in (\ref{eq:defT})--(\ref{eq:defE}),
Definition~\ref{dfn:K}, (\ref{eq:T1})--(\ref{eq:T2}),
(\ref{eq:Q1})--(\ref{eq:Q3}), and (\ref{eq:B0}).

A quadruple $B\in\binom{A}{4}$ is
said to be {\em symmetric} if $\orb{\hA}{B}=\orb{A}{B}$.

\begin{lem}\label{lem:symb}
A quadruple $B\in\binom{A}{4}$ is symmetric if and only 
if $\orb{\hA}{B}\in\cQ'\cup\cQ''\cup\cQ'''$, where
\begin{align}
\cQ'&=\{\orbt{a,b,a+b}\mid \{0,a,b,a+b\}\in\binom{A}{4}\},
\label{eq:lemsymb1}\\
\cQ''&=\{\orbt{a,-a,h}\mid \{0,a,-a,h\}\in\binom{A}{4},\;
2h=0\},\label{eq:lemsymb2}\\
\cQ'''&=\{\orbt{h,h',h''}\mid \{h,h',h''\}\in\binom{\Omega_1(A)\setminus\{0\}}{3}\}.\label{eq:lemsymb3}
\end{align}
In particular, every member of $\cB_0$ is symmetric.
Moreover, $\orb{\hA}{B}\in\cE$ implies that $B$ is symmetric.
\end{lem}
\begin{proof}
Suppose $B=\{0,a,b,a+b\}\in\binom{A}{4}$. Then $B^\sigma
=B-(a+b)\in\orb{A}{B}$. Thus $\orb{\hA}{B}\in\cQ'$ implies
that $B$ is symmetric.

Suppose $B=\{0,a,-a,h\}\in\binom{A}{4}$, $2h=0$. Then $B^\sigma
=B\in\orb{A}{B}$. Thus $\orb{\hA}{B}\in\cQ''$ implies
that $B$ is symmetric.

Suppose $B=\{0,h,h',h''\}\in\binom{A}{4}$, $\{h,h',h''\}\in\binom{\Omega_1(A)\setminus\{0\}}{3}$. 
Then $B^\sigma=B\in\orb{A}{B}$. Thus $\orb{\hA}{B}\in\cQ'''$ 
implies that $B$ is symmetric.

Conversely, suppose that $B=\{0,a,b,c\}\in\binom{A}{4}$
is symmetric. Then
\begin{align*}
\{-a,-b,-c\}\in&\{\{a,b,c\},\{-a,b-a,c-a\},
\\ &\quad
\{-b,a-b,c-b\},\{-c,a-c,b-c\}\},
\end{align*}
hence
\begin{align*}
&\{-a,-b,-c\}=\{a,b,c\}\text{ or}\\
&\{-b,-c\}=\{b-a,c-a\}\text{ or}\\
&\{-a,-c\}=\{a-b,c-b\}\text{ or}\\
&\{-a,-b\}=\{a-c,b-c\}.
\end{align*}
Observe
\begin{align*}
&\{-a,-b,-c\}=\{a,b,c\}
\\ &\iff
2a=0\text{ and }\{b,c\}=\{-b,-c\}\text{ or}\\
&\qquad\quad-a=b\text{ and }\{a,c\}=\{-b,-c\}\text{ or}\\
&\qquad\quad-a=c\text{ and }\{a,b\}=\{-b,-c\}\text{ or}
\nexteqv
2a=2b=2c=0\text{ or }(2a,b)=(0,-c)\text{ or}\\
&\qquad\quad-a=b\text{ and }2c=0\text{ or}\\
&\qquad\quad-a=c\text{ and }2b=0
\nexteqv
2a=2b=2c=0\text{ or }\\
&\qquad\quad(2a,b)=(0,-c)\text{ or}\\
&\qquad\quad(2b,c)=(0,-a)\text{ or}\\
&\qquad\quad(2c,a)=(0,-b)
\\ &\implies
\orb{\hA}{B}\in\cQ''\cup\cQ'''.
\end{align*}
Also, 
\begin{align*}
&\{-b,-c\}=\{b-a,c-a\}
\\ &\iff
(-b,-c)=(b-a,c-a)\text{ or}\\
&\qquad\quad(-b,-c)=(c-a,b-a)
\nexteqv
B=\{0,b,-b,c-b\}+b\text{ and }2(c-b)=0\text{ or}\\
&\qquad\quad a=b+c
\\ &\implies
\orb{\hA}{B}\in\cQ'\cup\cQ''.
\end{align*}
Similarly, 
$\{-a,-c\}=\{a-b,c-b\}$ or $\{-a,-b\}=\{a-c,b-c\}$
implies $\orb{\hA}{B}\in\cQ'\cup\cQ''$.

It follows from the definitions (\ref{eq:Q1}), (\ref{eq:Q2}), (\ref{eq:Q3}),
(\ref{eq:lemsymb1}), (\ref{eq:lemsymb2})
that $\cQ_1\subset\cQ''$, $\cQ_2\cup\cQ_3\subset\cQ'$. Thus
by (\ref{eq:B0}), 
we have $\orb{\hA}{B}\in\cQ'\cup\cQ''$ for any
$B\in\cB_0$. This implies that
every member of $\cB_0$ is symmetric.

Finally, as $\cE\subset\cQ'$, $\orb{\hA}{B}\in\cE$ implies
that $B$ is symmetric.
\end{proof}

We remark that, if $v\equiv2\pmod{4}$, then
it is shown in \cite{AMMS}
that the set of all symmetric blocks
forms an $A$-invariant $3$-$(v,4,3)$ design. 

\begin{lem}\label{lem:Kl1}
If $(A,\cB)$ is an $A$-reversible $\SQS(v)$ such that
$\cB_0\subset\cB$, then 
$\cB\setminus\cB_0\subset\{B\mid\orb{\hA}{B}\in\cE\}$.
\end{lem}
\begin{proof}
Let $B\in\cB$. Since $B$ is symmetric,
$\orb{\hA}{B}\in\cQ'\cup\cQ''\cup\cQ'''$ by Lemma~\ref{lem:symb}.

First suppose
$\orb{\hA}{B}\in\cQ''$. Then we may assume without loss
of generality $B=\{0,a,-a,h\}$ for some $a\in A$ and
$h\in\Omega_1(A)\setminus\{0\}$.
Since $\{0,a,-a\}\subset\{0,a,-a,h_0\}\in\cB_0\subset\cB$,
we obtain $B=\{0,a,-a,h_0\}\in\cB_0$.

Next suppose
$\orb{\hA}{B}\in\cQ'''$. Then we may assume without loss
of generality $B=\{0,h,h',h''\}$ for some 
$h,h',h''\in\Omega_1(A)\setminus\{0\}$.
Since $\{0,h,h'\}\subset\{0,h,h',h+h'\}\in\cB_0\subset\cB$,
we obtain $B=\{0,h,h',h+h'\}\in\cB_0$.

Finally, suppose $\orb{\hA}{B}\in\cQ'$.
It suffices to show that 
$\orb{\hA}{B}\notin\cE$ implies $B\in\cB_0$.
We may assume without loss
of generality $B=\{0,a,b,a+b\}$ for some $a,b\in A$.
Since $\orb{\hA}{B}\notin\cE$, Lemma~\ref{lem:cE}(ii) implies
$0\in\{2a,2b\}$ or $\{\pm a,\pm2a\}\cap\{\pm b,\pm2b\}
\neq\emptyset$. If the former occurs, then $B$ contains
a triple $T$ with $\orb{\hA}{T}\in\cT_2$.
If the latter occurs with $0 \notin \{2a, 2b\}$, then
we may assume $2a\in\{b,2b\}$ by replacing
$a$ by $\pm b$ if necessary. If $2a=b$, then $B-a$
contains $T=\{0,a,-a\}$, and $\orb{\hA}{T}\in\cT_1$.
If $2a=2b$, then $B-a$ contains the triple
$T=\{0,-a,b-a\}$, and we have $\orb{\hA}{T}\in\cT_2$ since $2(b-a)=0$.
Therefore, we have shown that there exists $T\subset B$
such that $\orb{\hA}{T}\in\cT_1\cup\cT_2$. 
Note that the existence of an $\SQS(v)$ implies 
$v\equiv2$ or $4\pmod{6}$, so we can apply
Lemma~\ref{lem:ps1} to conclude that there exists
$B'\in\cB_0\subset\cB$ such that $T\subset B'$.
This forces $B=B'\in\cB_0$.
\end{proof}

\begin{thm}\label{thm:Kohler}
Let $A$ be an abelian group of order $v\equiv2$ or $4\pmod6$.
For a subset $\cB$ of $\binom{A}{4}$ containing $\cB_0$,
the incidence structure $(A,\cB)$ is an $A$-reversible
$\SQS(v)$ if and only if 
\[
\cB=\cB_0\cup\{B\in\binom{A}{4}\mid\orb{\hA}{B}\in\cF\}
\]
for some $1$-factor $\cF$ of the K\"ohler graph of $A$.
\end{thm}
\begin{proof}
By Lemma~\ref{lem:ps1}, we have
\begin{equation}\label{eq:thmK1}
|\{B\in\cB_{0}\mid T\subset B\}|=1\text{ for }
\forall T\in\binom{A}{3}\text{ with }\orb{\hA}{T}\in\cT_1\cup\cT_2,
\end{equation}
and also, together with Lemma~\ref{lem:T}, we have
\begin{equation}\label{eq:thmK2}
|\{B\in\cB_0\mid T\subset B\}|=0\text{ for }
\forall T\in\binom{A}{3}\text{ with }\orb{\hA}{T}\in\cT.
\end{equation}
Thus
\begin{align*}
&\text{$(A,\cB)$ is an $A$-reversible $\SQS(v)$}
\\ &\iff
\text{$(A,\cB)$ is an $A$-reversible $\SQS(v)$,}
\\ &\qquad\quad
\cB\setminus\cB_0\subset\{B\mid\orb{\hA}{B}\in\cE\}
&&\text{(by Lemma~\ref{lem:Kl1})}
\nexteqv
\text{$\cB$ is $\hA$-invariant,}\\
&\qquad\quad\text{every member of $\cB$ is symmetric,}\\
&\qquad\quad |\{B\in\cB\mid T\subset B\}|=1\text{ for }
\forall T\in\binom{A}{3},
\\ &\qquad\quad
\cB\setminus\cB_0\subset\{B\mid\orb{\hA}{B}\in\cE\}
\nexteqv
\text{$\cB$ is $\hA$-invariant,}\\
&\qquad\quad |\{B\in\cB\mid T\subset B\}|=1\text{ for }
\forall T\in\binom{A}{3},
\\ &\qquad\quad
\cB\setminus\cB_0\subset\{B\mid\orb{\hA}{B}\in\cE\}
&&\text{(by Lemma~\ref{lem:symb})}
\nexteqv
\text{$\cB\setminus\cB_0$ is $\hA$-invariant,}\\
&\qquad\quad
\cB\setminus\cB_0\subset\{B\mid\orb{\hA}{B}\in\cE\},\\
&\qquad\quad |\{B\in\cB\mid T\subset B\}|=1\text{ for }
\forall T\in\binom{A}{3}
\nexteqv
\cB\setminus\cB_0=\{B\mid\orb{\hA}{B}\in\cF\}
\text{ for some }\cF\subset\cE,\\ 
&\qquad\quad |\{B\in\cB\mid T\subset B\}|=1\text{ for }
\forall T\in\binom{A}{3}
\nexteqv
\cB\setminus\cB_0=\{B\mid\orb{\hA}{B}\in\cF\}
\text{ for some }\cF\subset\cE,\\ 
&\qquad\quad |\{B\in\cB\setminus\cB_0\mid T\subset B\}|=1\\
&\qquad\quad \text{ for }
\forall T\in\binom{A}{3}\text{ with }\orb{\hA}{T}\in\cT
&&\text{(by (\ref{eq:thmK1}),(\ref{eq:thmK2}))}
\nexteqv
\cB\setminus\cB_0=\{B\mid\orb{\hA}{B}\in\cF\}
\text{ for some }\cF\subset\cE,
\\ &\qquad\quad
|\{B\mid\orb{\hA}{B}\in\cF,\;T\subset B\}|=1
\\ &\qquad\quad
\text{ for }
\forall T\in\binom{A}{3}\text{ with }\orb{\hA}{T}\in\cT.
\end{align*}
Now, for $T\in\binom{A}{3}$ with $\orb{\hA}{T}\in\cT$,
\begin{align*}
&|\{B\mid\orb{\hA}{B}\in\cF,\;T\subset B\}|
\\ &=
|\{B\mid\orb{\hA}{B}\in\cF,\text{ $\orb{\hA}{T}$
is incident with $\orb{\hA}{B}$}\}|
\nexteq
|\{\orb{\hA}{B}\in\cF\mid\text{$\orb{\hA}{T}$
is incident with $\orb{\hA}{B}$}\}|
\end{align*}
by Lemma~\ref{lem:onlyB}.
Therefore,
\begin{align*}
&\text{$(A,\cB)$ is an $A$-reversible $\SQS(v)$}
\\ &\iff
\cB\setminus\cB_0=\{B\mid\orb{\hA}{B}\in\cF\}
\text{ for some $1$-factor }\cF\subset\cE.
\end{align*}
\end{proof} 

\begin{rem}\label{rem:Kohler}
We note that not all $A$-reversible $\SQS(v)$ are
constructed using 
Theorem~\ref{thm:Kohler}. 
Let
$A = \langle h_0\rangle \oplus \langle
 h \rangle \oplus \langle c \rangle$, where
both $h_0$ and $h$ have order $2$, and $c$ has order $5$. 
Then, an $A$-reversible $\SQS(20)$ can be constructed 
by the quadruples given in Table~\ref{tbl-1}.
We note that for the block 
$B=\{0,c,-c+h,h\}$, we have
$B\notin\cB_0$ and $\orb{\hA}{B}\notin\cE$.
\end{rem}
\begin{table}[h]
\begin{center}
\begin{tabular}{|l|} \hline
$[h_0,h,h_0+h]$\\
\hline
$[c,-c,h_0]$\\
$[2c,-2c,h_0]$\\
$[c+h,-c+h,h_0]$\\
$[2c+h,-2c+h,h_0]$\\
\hline
$[h_0+h,c+h_0,c+h]$\\
$[h_0+h,c,c+h_0+h]$\\
$[h_0+h,2c,2c+h_0+h]$\\
$[h_0+h,2c+h_0,2c+h]$\\
$[h,c+h_0,c+h_0+h]$\\
$[h,2c+h_0,2c+h_0+h]$\\
\hline
$[c+h_0,-2c+h_0,-c]$\\
$[c+h_0+h,-2c+h_0+h,-c]$\\
$[2c+h_0,c+h_0+h,-2c+h]$\\
$[2c+h_0+h,c+h,-2c+h_0]$\\
\hline
$[c,-c+h,h]$\\
$[2c,-2c+h,h]$\\
\hline
\end{tabular}
\end{center}
\caption{$\SQS(20)$ with $\cB_0\not\subset\cB$}
\label{tbl-1}
\end{table}

\begin{thm}\label{thm:Series2n}
Let $n\ge3$ be an integer, and let
$A$ be an abelian group of order $2^n$ whose
exponent is $2$ or $4$. Then
there exists an $A$-reversible $\SQS(2^n)$.
\end{thm}
\begin{proof}
By Theorem~\ref{thm:Kohler},
it suffices to show that the 
K\"ohler graph of $A$ has a $1$-factor.
By Lemma~\ref{lem:2gen2}, 
it suffices to show that the 
K\"ohler graph of
$\Z_2^{\epsilon_1} \oplus \Z_4^{\epsilon_2}$ 
has a $1$-factor whenever $\epsilon_1+\epsilon_2\leq2$.
The K\"ohler graph of
$\Z_2^{\epsilon_1} \oplus \Z_4^{\epsilon_2}$, 
where $\epsilon_1+\epsilon_2\leq2$
is empty unless $(\epsilon_1,\epsilon_2)=(0,2)$.
By Example~\ref{exam:Z4Z4},
the K\"ohler graph of $\Z_4^2$ is a $3$-cube
and hence has a $1$-factor.
\end{proof}

\section{Abelian groups with cyclic Sylow $2$-subgroup}

As in Section~2, we let
$A$ be an abelian group of order $v$, and use the notation
introduced in Section~2.
Moreover, in this section, we assume that the Sylow $2$-subgroup
of $A$ is cyclic.

\begin{lem} \label{lem:noncycle3:n=1} 
Suppose $a,b\in A$, and 
$\langle a,b\rangle$ is not cyclic. 
Then (\ref{eq:2ab}) holds, and
$\orbt{a,b}$ is a vertex of degree $3$
in the K\"ohler graph of $A$.
\end{lem}
\begin{proof}
Suppose contrary, that $2a\in\langle b\rangle$.
Let $S$ denote the Sylow $2$-subgroup of $A$.
Since $a+ \langle b\rangle$ has order $2$ in 
$A/ \langle b\rangle$, 
it belongs to the Sylow $2$-subgroup 
$(S+\langle b\rangle)/\langle b\rangle$ of
$A/ \langle b\rangle$.
Hence we have $a \in S + \langle b\rangle$. Since $S$ is the
Sylow $2$-subgroup which is cyclic, $S+\langle b\rangle$ is also
cyclic. This implies that $\langle a, b\rangle$ is cyclic,
contradicting the assumption. This proves (\ref{eq:2ab}).

It follows from (\ref{eq:defT}) that $\orbt{a,b}$ is a 
vertex of $\cG$. Also, it follows from Lemma~\ref{lem:orbT3}
that $\orbt{a,b}$ has degree $3$.
\end{proof}

\begin{rem}
In general, there may be vertices of degree less than $3$ 
in the K\"ohler graph of an abelian group $A$, if the Sylow
$2$-subgroup is not cyclic.
Indeed, let
$A=\langle x\rangle\oplus\langle y\rangle=\Z_2\oplus\Z_{2m}$ for some
$m\ge3$. Let $a=x+y$, $b=-y$. Then $a\ne\pm b$, $2a=2y\notin
\{0,-y,-2y\}=\{0,b,2b\}$, and $2b=-2y\notin\{0,x+y,2y\}
=\{0,a,2a\}$. Thus $\orbt{a,b}\in\cT$ and $2a+2b=0$. It
follows from
Lemmas~\ref{lem:nbs} and \ref{lem:orbT3}
that the vertex $\orbt{a,b}$ of
the K\"ohler graph of $A$ has degree less than $3$.
\end{rem}

\begin{lem}\label{lem:3reg2ec}
Suppose $a,b\in A$, and 
$\langle a,b\rangle$ is not cyclic. 
Then
the connected component containing $\orbt{a,b}$,
of the K\"ohler graph of $A$ 
is $3$-regular and $2$-edge-connected, 
and in particular has a $1$-factor.
\end{lem}
\begin{proof}
First note that, by Lemma~\ref{lem:noncycle3:n=1},
$\orbt{a,b}$ is indeed a vertex of $\cG$.

Suppose that $\orbt{c,d}\in\cT$ is a vertex in the connected
component $\cC$ of the K\"ohler graph of $A$ containing $\orbt{a,b}$.
By Lemma~\ref{lem:2gen}, we have $\langle a,b\rangle=\langle c,d\rangle$,
hence $\orbt{c,d}$ has degree $3$ by Lemma~\ref{lem:noncycle3:n=1}.
This proves that $\cC$ is $3$-regular.

To prove the 2-edge-connectivity, it suffices to
find a cycle containing a given edge. 
Without loss of generality, we may assume that a given edge
is incident with the vertex $\orbt{a,b}$.
Then by Lemma~\ref{lem:nbs},
there are three possibilities:
\[
\orbt{a,b,a+b},\orbt{a,b,a-b},\orbt{a,b,b-a}.
\]
In the latter two cases, we may replace $(a,b)$ by $(a-b,b)$,
$(a,b-a)$, respectively, to reduce to the first case. 
Since $\langle a,b\rangle$ is not cyclic, 
(\ref{eq:2ab}) holds by Lemma~\ref{lem:noncycle3:n=1},
and hence by Lemma~\ref{lem:edgecyc}, there exists a cycle
containing $\orbt{a,b,a+b}$.
We thus conclude that $\cC$ is $2$-edge-connected. 
It is well known in graph theory (\cite{Petersen}, 
see also~\cite{LL}[p.59])
that any $2$-edge-connected and
$3$-regular graph has a $1$-factor, which completes the proof. 
\end{proof}

Now let $A$ be an abelian group of order $v\equiv2$ or $4\pmod6$.
We fix an element $h_0\in A$ of order $2$, and 
use the notation introduced in (\ref{eq:T1})--(\ref{eq:T2}),
(\ref{eq:Q1})--(\ref{eq:Q3}), and (\ref{eq:B0}).

\begin{lem}\label{lem:cycS2}
Let $A$ be an abelian group of order $v$ whose Sylow
$2$-subgroup is cyclic.
If $(A,\cB)$ is an $A$-reversible $\SQS(v)$, then
$\cB_0\subset\cB$.
\end{lem}
\begin{proof}
Observe $\cQ_2=\cQ_3=\emptyset$. Thus
it suffices to show $\{0,a,-a,h_0\}\in\cB$ for any
$a\in A\setminus \Omega_1(A)$. 
Let $B\in\cB$ be the unique block containing $\{0,a,-a\}$.
Then $B=\{0,a,-a,b\}$ for some $b\in A$. 
Since $\{0,a,-a\}\subset B\cap (-B)$ and $-B\in\cB$,
we must have $B=-B$, which implies $b$ has order $2$.
Since the Sylow $2$-subgroup of $A$ is cyclic, $h_0$ is
the unique element of order $2$, hence $b=h_0$.
\end{proof}

\begin{thm}\label{thm:cycS2}
Let $A$ be an abelian group of order $v\equiv2$ or $4\pmod6$
such that the Sylow $2$-subgroup of $A$ is cyclic. Then
there exists an $A$-reversible $\SQS(v)$ if and
only if the K\"ohler graph of $A$ has a $1$-factor.
\end{thm}
\begin{proof}
Immediate from Theorem~\ref{thm:Kohler} and Lemma~\ref{lem:cycS2}.
\end{proof}

\section{Main theorem}

The following theorem is an extension of a theorem of
\cite{WP}.

\begin{thm} \label{thm:Series}
Let $v\ge8$ be a positive integer.
The following statements are equivalent: 
\begin{enumerate}
\item There exists an $A$-reversible $\SQS(v)$ for any abelian group $A$ of
order $v$ whose Sylow $2$-subgroup is cyclic;
\item There exists an $A$-reversible $\SQS(v)$ for some
abelian group $A$ of order $v$ whose Sylow $2$-subgroup is cyclic;
\item There exists an $S$-cyclic $\SQS(v)$;
\item 
$v\equiv 0\pmod{2}$, $v\not\equiv0\pmod{3}$,
$v\not\equiv0\pmod{8}$, and 
there exists an $S$-cyclic $\SQS(2p)$
for any odd prime divisor $p$ of $v$.
\end{enumerate}
\end{thm}
\begin{proof}
Clearly, (i) implies (ii) and (iii).
The equivalence of (iii) and (iv) is due to
\cite{WP}[Satz 14.1].
So it remains
to prove (ii)$\implies$(iv) and (iv)$\implies$(i).

Suppose (ii) holds. Let $A$ be an abelian group
of order $v$ whose Sylow $2$-subgroup is cyclic,
and suppose that there exists an $A$-reversible $\SQS(v)$.
This implies 
$v\equiv 0\pmod{2}$ and $v\not\equiv0\pmod{3}$.
By Theorem~\ref{thm:cycS2}, the K\"ohler graph of $A$ has
a $1$-factor. Since $A$ has a cyclic subgroup of order $2p$
for any odd prime divisor $p$ of $v$, Lemma~\ref{lem:2gen2}
implies that the K\"ohler graph of a cyclic group of order
$2p$ has a $1$-factor. 
By Theorem~\ref{thm:cycS2} again, there exists a
$\Z_{2p}$-reversible $\SQS(2p)$.
Also, as the K\"ohler graph of a cyclic
group of order $8$ has no $1$-factor, $A$ has no element of
order $8$. Since the Sylow $2$-subgroup of $A$ is cyclic,
this implies $v\not\equiv0\pmod{8}$. 
Therefore, (iv) holds.

Next we prove (iv)$\implies$(i).
Let $A$ be an arbitrary abelian group of order $v$
whose Sylow $2$-subgroup is cyclic.
In view of Theorem~\ref{thm:cycS2}, it suffices to show that
the K\"ohler graph of $A$ has a $1$-factor. 
Let $\orbt{a,b}$ be a vertex of the K\"ohler graph $\cG$
of $A$.
The connected component of $\cG$ containing $\orbt{a,b}$
is isomorphic to a connected component of the K\"ohler graph of 
$A'=\langle a,b\rangle$, by Lemma~\ref{lem:2gen2}. 

If $A'$ is not cyclic, then by Lemma~\ref{lem:3reg2ec},
the connected component of $\cG$ containing $\orbt{a,b}$
has a $1$-factor.

Suppose that $A'$ is cyclic.
If $A'$ has odd order, then there exists a cyclic
subgroup $\tilde{A}$
containing $A'$ with $|\tilde{A}:A'|=2$.
By the implication (iv)$\implies$(iii),
there exists an $S$-cyclic $\SQS(2|A'|)$.
It follows from
Theorem~\ref{thm:cycS2} that there exists
a $1$-factor in the K\"ohler graph of $\tilde{A}$.
By Lemma~\ref{lem:2gen2}, there exists 
a $1$-factor in the K\"ohler graph of $A'$,
and in particular, there exists a $1$-factor 
in the connected component of $\cG$ containing $\orbt{a,b}$.
\end{proof}

In~\cite{HS2},
it is shown that a $S$-cyclic $\SQS(2p)$ exists for
any prime number $p\equiv53$ or $77\pmod{120}$ with $p<500000$.
Applying Theorem~\ref{thm:Series} to this result shows that there exists an
$A$-reversible $\SQS(v)$
for any abelian group $A$ of order $v$ which is twice a
product of prime numbers $p$ with
$p\equiv53$ or $77\pmod{120}$
and $p<500000$. 
The interested reader is also referred to~\cite{OP2,MH} for
recent results on SQS with various automorphism groups.

\bibliographystyle{jstpip}
\bibliography{MunemasaSawa}

\end{document}